\documentclass[11pt,reqno,a4paper]{amsart}
\usepackage{amsmath,amsthm,verbatim,amscd,amssymb,setspace,enumitem,hyperref}
\usepackage{exscale,color}

\usepackage{enumitem}
\usepackage{mathrsfs}

\renewcommand{\epsilon}{\varepsilon}
\newcommand{\N}{\mathbb{N}}
\newcommand{\Z}{\mathbb{Z}}

\newcommand{\R}{\mathbb{R}}
\newcommand{\C}{\mathbb{C}}

\renewcommand{\Re}{\operatorname{Re}}

\newcounter{mtheorem}
\newtheorem{mtheorem}[mtheorem]{Theorem}

\newcommand{{\vol}}{\rm vol}

\newcommand{\Ric}{\operatorname{Ric}}

\newcommand{\Rm}{\operatorname{Rm}}

\def\tr{\operatorname{tr}}

\newtheoremstyle{fancy}{}{}{\itshape}{}{\textbf\bgroup}{.\egroup}{ }{}
\newtheoremstyle{fancy2}{}{}{\rm}{}{\textbf\bgroup}{.\egroup}{ }{}

\theoremstyle{fancy}
\newtheorem{theorem}{Theorem}[section]
\newtheorem{lemma}[theorem]{Lemma}
\newtheorem{corollary}[theorem]{Corollary}

\newtheorem{prop}[theorem]{Proposition}

\theoremstyle{fancy2}
\newtheorem{definition}[theorem]{Definition}

\newtheorem{remark}[theorem]{Remark}

\newtheorem{claim}[theorem]{Claim}

\newtheorem{ques}[theorem]{Question}

\setlist{leftmargin=*}

\numberwithin{equation}{section}

\begin{document}
\title{Uniqueness of asymptotically conical K\"ahler-Ricci flow}
\date{\today}
\author{Longteng Chen}
\address{Université Paris-Saclay, CNRS, Laboratoire de Mathématiques d'Orsay, 91405 Orsay, France }
\email{longteng.chen@universite-paris-saclay.fr}
\begin{abstract}
   We study the uniqueness problem for the Kähler-Ricci flow with a conical initial condition. Given a complete gradient expanding Kähler-Ricci soliton on a noncompact manifold with quadratic curvature decay, including its derivatives, we establish that any complete solution to the Kähler-Ricci flow emerging from the soliton's tangent cone at infinity—appearing as a Kähler cone—must coincide with the forward self-similar Kähler-Ricci flow associated with the soliton, under the stated hypotheses. In
particular, the curvature hypothesis
requires only a $C/t$ bound on the Ricci
tensor. This paper provides an answer to the question of the uniqueness of desingularization of Kähler cones via the Kähler-Ricci flow.
\end{abstract}
\maketitle
\markboth{Longteng Chen}{Uniqueness of asymptotically conical K\"ahler-Ricci flow}
\section{Introduction}\label{Introduction}
\subsection{Overview}
A smooth family $g(t)_{t\in (0,T)}$ of K\"ahler metrics on a complex manifold $M$ solves the K\"ahler-Ricci flow if
\begin{equation}\label{KahlerRicciFlowEq}
    \frac{\partial}{\partial t} g(t) = -\Ric(g(t)), \quad t \in (0,T).
\end{equation}

 An \emph{expanding K\"ahler-Ricci soliton} is a triple $(M,g,X)$, where $M$ is a complex manifold with complex structure $J$ and a complete K\"ahler metric $g$ and a complete real-holomorphic vector field $X$ satisfying the equation
\begin{equation}\label{soliton2}
\frac{1}{2}\mathcal{L}_{X}g=\Ric(g)+ g.
\end{equation}
If $X=\nabla^{g} f$ holds for some real-valued smooth function $f$ on $M$,
then we say that $(M,\,g,\,X)$ is \emph{gradient}. In this case, the soliton equation \eqref{soliton2}
reduces to\begin{equation}\label{soliton1}
\operatorname{Hess}_{g}(f)=\Ric(g)+g,
\end{equation}
if $\omega$ is the K\"ahler form of $g$ and $\rho_{\omega}$ is the Ricci form of $\omega$, we rewrite \eqref{soliton1} as:
\begin{equation}\label{krseqn}
i\partial\bar{\partial}f=\rho_{\omega}+\omega.
\end{equation}
For a gradient K\"ahler-Ricci soliton $(M,g,X)$, the vector field $X$ is called the
\emph{soliton vector field}. Its completeness is guaranteed by the completeness of $g$
\cite{Z-H.Zhang}. If the soliton is gradient, then
the smooth real-valued function $f$ satisfying $X=\nabla^g f$ is called the \emph{soliton potential}. It is unique up to a constant.

Every expanding K\"ahler-Ricci soliton is related to a self-similar solution to K\"ahler-Ricci flow. Indeed, if $(M,g,X)$ is a complete expanding K\"ahler-Ricci soliton, we use $(\Phi^X_{\tau},\tau\in \R)$ to denote the flow of the complete, real-holomorphic vector field $X$. We define the biholomorphism $\Phi_t: M\mapsto M$ such that $\Phi_t=\Phi^X_{-\frac{1}{2}\log t}$ holds for all $t>0$. Then $g(t):=t\Phi_t^*g$ is a solution to K\"ahler-Ricci flow. 

Given now an \emph{expanding} gradient K\"ahler-Ricci soliton $(M,g,X)$ with quadratic Riemannian curvature decay, including its derivatives, we can associate a unique asymptotic K\"ahler cone $(C,g_0)$ (see \cite{Conlon-Deruelle-Sun,Conlon-Deruelle}, for concrete examples, see \cite{Feldman-Ilmanen-Knopf}), which is then an initial condition of the self-similar K\"ahler-Ricci flow $g(t),t\ge 0$ associated to the soliton in the sense $\lim_{t\to0^+}g(t)=g_0$ as a Gromov-Hausdorff limit \cite[Theorem A]{Conlon-Deruelle-Sun}. The work of this paper is motivated by the uniqueness of K\"ahler-Ricci flow in this setting, namely, the following question:

\begin{ques}\label{question of FIK}
 Let $(C,g_0)$ be a Kähler cone that can be desingularized both by a self-similar solution to the Kähler–Ricci flow and by a generic solution to the Kähler–Ricci flow. Do these two solutions necessarily coincide?
\end{ques}

In \cite{Conlon-Deruelle-Sun}, Conlon, Deruelle and Sun established that a complete expanding gradient Kähler–Ricci soliton $(M,g,X)$ with quadratic curvature decay (including all derivatives) desingularizing a Kähler cone $(C,g_0)$ exists and is unique (up to pullback by biholomorphism) if and only if the cone admits a smooth canonical model. In this situation, there exists a smooth Kähler resolution $\pi : M \mapsto C$ with exceptional set $E$, such that the pushforward metric $\pi_* g(t)$ converges locally smoothly to $g_0$ as $t \to 0^+$, where $g(t)$, $t>0$, denotes the self-similar solution associated to the soliton. In this case, we call $(M,g,X)$ an asymptotically conical (AC) gradient K\"ahler-Ricci expander. For a formal definition, see Definition \ref{ACKR expander}.
 
    If $g(t):=t\Phi_t^*g$ is the self-similar forward K\"ahler-Ricci flow associated with $g$, since the full Riemann curvature tensor of $g$ is bounded, there is a uniform constant $A>0$ such that for all $t>0$, we have
    \begin{equation*}
        \sup_M |\Rm(g(t))|_{g(t)}\le \frac{A}{t}.
    \end{equation*}
  Consequently, $(M,g(t))_{t>0}$ is an immortal Ricci flow that becomes possibly singular as $t\to 0^+$.
  
 In this paper, we prove that there is only one way to desingularize the K\"ahler cone $(C,g_0)$ with K\"ahler-Ricci flow provided a number of conditions hold; that is the self-similar forward K\"ahler-Ricci flow associated with the given expanding gradient K\"ahler-Ricci soliton.
\subsection{Main result}
We regard \eqref{KahlerRicciFlowEq} as a PDE system with singular initial data. In this paper, we give a partial answer to Question \ref{question of FIK} for a K\"ahler cone with a smooth canonical model. More precisely, we study the uniqueness property of K\"ahler-Ricci flow on an AC gradient K\"ahler-Ricci expander (see Definition \ref{ACKR expander}) with asymptotic K\"ahler cone.

For the analysis of the Ricci flow with singular initial data, it is customary to impose that the flow admits the same $A/t$ curvature bound as the self-similar solution. In contrast, as an improvement over the existing theory, we shall instead consider the Kähler–Ricci flow with a more general curvature bound.
\begin{definition}[Type III$^\ast$ singularity]
Let $\tilde g(t)_{t\in (0,T)}$ be a smooth solution to the Ricci flow on a manifold $M$. We say that $\tilde g(t)$ has a singularity of type III$^\ast$ if the following two conditions hold:
\begin{enumerate}
    \item (Local boundedness of full curvature) For every compact time interval $I \subset (0,T)$, the full Riemann curvature tensor remains uniformly bounded on $M\times I$, i.e.,
   $$
   \sup_{M\times I}|\Rm(\tilde g(\cdot))|_{\tilde g(\cdot)} < +\infty.  
   $$
\item (Global Ricci bound) There exists a constant $A>0$ such that for all $t\in (0,T)$,
   $$
   \sup_{M}|\Ric(\tilde g(t))|_{\tilde g(t)} \;\le\; \frac{A}{t},  
   $$
   where $\Ric(\tilde g(t))$ denotes the Ricci curvature tensor of $\tilde g(t)$.
\end{enumerate}
Any Ricci flow $(\tilde g(t))_{t\in(0,T)}$ satisfying an $A/t$ curvature bound is clearly of Type III$^\ast$ singularity.

   \end{definition}
Our main theorem is as follows:
\begin{mtheorem}[Uniqueness theorem]\label{main}
Let $(M,g,X)$ be an AC gradient K\"ahler-Ricci expander as in Definition \ref{ACKR expander} with asymptotic K\"ahler cone $(C,g_0)$. Let $g(t):=t\Phi_t^*g$ be the self-similar forward K\"ahler-Ricci flow associated to $g$ and let $\omega(t)$ be its K\"ahler form. Let $\pi: M\mapsto C$ be the K\"ahler resolution as in Theorem \ref{CDS theorem}.

Assume that $\varphi\in C^\infty(M\times (0,T))$ for some $T>0$ is a real-valued smooth function such that
\begin{equation*}
    \omega_\varphi(t):=\omega(t)+i\partial\bar\partial \varphi(t),\quad t\in (0,T)
\end{equation*} defines a smooth complete solution to the Kähler–Ricci flow with type III$^\ast$ singularity, and that
    \begin{enumerate}
     \item \textnormal{(Conical condition)} $\pi_* g_\varphi(t)$ converges to $g_0$ locally smoothly outside the apex $o$ as $t\to0^+$. Here $g_\varphi(t)$ denotes the Riemannian metric associated to $\omega_\varphi(t)$.
    \item\textnormal{(Symmetry condition)} The Reeb vector field $JX$ is a Killing vector field for $g_\varphi(t_0)$ for some $t_0\in (0,T)$. Here $J$ denotes the complex structure of $M$.
    \end{enumerate}
    Then $g_\varphi(t)=g(t)$ for all $t\in (0,T)$.
\end{mtheorem}

From the conical condition in Theorem \ref{main}, the two solutions to K\"ahler-Ricci flow $(g(t))_{t\in (0,+\infty)}$ and $(g_\varphi(t))_{t\in (0,T)}$ are both solutions to K\"ahler-Ricci flow starting with the same K\"ahler cone $(C,g_0)$. The main difficulties in proving Theorem \ref{main} arise from two sources: the non-compactness of the underlying complex manifold and the presence of singular initial data. Since the convergence we obtain is only local, away from the singular set, establishing the uniqueness of the Kähler–Ricci flow in this setting becomes highly nontrivial.
\begin{remark}\label{rmk}
In Theorem \ref{main}, we do not assume that the entire Riemann curvature tensor is bounded by $ \frac{A}{t} $ (which happens in many cases); we only need a precise bound for Ricci curvature. In Section \ref{Ricci flow coming out of cone}, we will show that having a bound on the full Riemann curvature operator for each compact time interval allows us to apply Perelman's pseudolocality theorem, which in turn implies that the full Riemann curvature operator and its derivatives decay quadratically. This fact ensures that the full Riemann curvature operator remains bounded outside a compact set---a sub-level set of soliton potential. Furthermore, we will show that the $ \frac{A}{t} $ bound on the Ricci curvature plays a crucial role in obtaining estimates within such a compact region.
\end{remark}
Since $M$ admits an expanding Kähler–Ricci soliton, we have $-\omega\in c_1(M)$. Let $\omega_1(t)$ be another Kähler–Ricci flow satisfying the conical condition in Theorem~\ref{main}. Since
\begin{equation*}
\frac{\partial}{\partial t}\omega_1(t)=-\rho_{\omega_1(t)}
\qquad\text{and}\qquad
\frac{\partial}{\partial t}\omega(t)=-\rho_{\omega(t)},
\end{equation*}
and both Ricci forms represent $c_1(M)$, we obtain
\begin{equation*}
\frac{\partial}{\partial t}\omega_1(t)-\omega(t)=i\partial\bar\partial \log\frac{\omega_1(t)^n}{\omega(t)^n}.
\end{equation*}
The conical condition implies that $\pi_*(\omega_1(0)-\omega(0))=0$, and hence $\omega_1(t)-\omega(t)=i\partial\bar\partial\left( \int_0^t\log\frac{\omega_1(s)^n}{\omega(s)^n}ds\right)$ holds on $\pi^{-1}(C\setminus\{o\})$ for every $t>0$. However, for singular initial metrics, we cannot conclude that $\omega_1(t)-\omega(t)$ is globally exact on $M$ unless the function $\int_0^t\log\frac{\omega_1(s)^n}{\omega(s)^n}ds$ can be extended smoothly on $M$, since the topology and the behavior of the flows along $E$ as $t\to0^+$ may depend on the singular initial data. 

We have already seen that on a compact complex manifold, the K\"ahler-Ricci flow equation can always be reduced to a complex Monge-Ampère equation. However, in the non-compact case, this correspondence does not always hold. In Section  \ref{Normalization of problem}, we will prove (see Lemma \ref{Killing condition}) that the vector field $JX$ is Killing for $g_\varphi(t)$ for all $t\in(0,T)$ by using the uniqueness theorem of Chen-Zhu \cite{Chen-Zhu} and the backward uniqueness theorem of Kotschwar \cite{Kotschwar3}. Moreover, this Killing vector field plays a crucial role in normalizing the problem, allowing us to reduce the Kähler-Ricci flow equation to a complex Monge-Amp\`ere equation as desired.
\subsection{History on the uniqueness of Ricci flow}
The Ricci flow equation was first introduced by R. Hamilton \cite{R.Hamilton}, and it is a central tool in differential geometry. Since it is a nonlinear degenerate parabolic equation, the uniqueness of solutions is an important problem.

For compact manifolds, Hamilton \cite{R.Hamilton} established uniqueness. On non-compact manifolds, Shi \cite{W-X.Shi} constructed complete solutions with bounded curvature from smooth, complete initial metrics. Chen and Zhu \cite{Chen-Zhu} proved that two complete solutions with bounded curvature and the same initial metric coincide, while Kotschwar \cite{Kotschwar3} established backward uniqueness: solutions with bounded curvature that agree at $t=T$ are identical for all $t$.

Kotschwar \cite{Kotschwar1,Kotschwar2} further obtained uniqueness under weaker curvature bounds of the form $\frac{C}{t^\delta}$ ($\delta\in[0,1)$) and metric equivalence conditions, assuming volume growth for the initial metric. However, his method does not cover the case $\frac{C}{t}$, which is scale-invariant: if $g(t)$ has curvature bounded by $\frac{C}{t}$, the dilated flow $g_\lambda(t) = \lambda g(\lambda^{-1}t)$ also satisfies the Ricci flow with the same curvature bound.

The sharpest result is due to Lee \cite{M-C.Lee}, who proved the uniqueness for complete solutions with initial metric $g$ and curvature bounded by $\frac{C}{t}$. In dimension 3, if $g$ has nonnegative curvature operator $0\le\Rm(g)\le K_0$, Chen \cite{B-L.Chen} showed that two complete solutions with initial metric $g$ coincide on $[0,\min\{T,\frac{1}{(4K_0)}\}]$. In particular, the Ricci flow starting from the Euclidean metric on $\mathbb{R}^3$ is static.

On the non-uniqueness side, Angenent and Knopf \cite{Angenent-Knopf} constructed, in dimensions $n\geq 5$, multiple expanding solitons emanating from the same Riemannian cone. Their examples also show that Theorem \ref{main} fails without the Kähler assumption. Whether such non-uniqueness can occur in the smooth space-time setting remains an open question.

 \subsection{Strategy of proof}
Recall that $g_\varphi(t)_{t\in (0,T)}$ and $g(t)_{t\in\mathbb{R}_+}$ are two solutions to the K\"ahler-Ricci flow starting with the same conical initial data as in Theorem \ref{main}. The strategy of this paper is to analyze the \emph{normalized Kähler-Ricci flow} and investigate whether it is a static flow.

The normalized Kähler-Ricci flow $ g_\psi(\tau)_{\tau\in (-\infty,\log T)} $ is defined through the relation
\begin{equation*}
    e^\tau \Phi_{e^\tau}^* g_\psi(\tau) = g_\varphi(e^\tau).
\end{equation*}
The evolution equation for $ g_\psi(\tau) $ then becomes
\begin{equation*}
    \partial_\tau g_\psi(\tau) = \mathcal{L}_{\frac{X}{2}} g_\psi(\tau) - \mathrm{Ric}(g_\psi(\tau)) - g_\psi(\tau).
\end{equation*}

This geometric flow equation can be further reduced to a \emph{complex Monge-Ampère equation} for a scalar potential function $\psi(\tau)$, which is given by:
\begin{equation}\label{eq: first appearance of psi}
  \psi(\tau)=e^{-\tau}\Phi_{e^{-\tau}}^*\varphi(e^\tau),\quad  \partial_\tau \psi(\tau) = \log\left(\frac{\omega_\psi(\tau)^n}{\omega^n}\right) + \frac{X}{2} \cdot \psi(\tau) - \psi(\tau),
\end{equation}
where $\omega_\psi(\tau)^n$(resp. $\omega^n$) denotes the volume form associated with the Kähler metric $g_\psi(\tau)$(resp $g$).

Taking the time derivative $\dot{\psi} := \partial_\tau \psi$, we obtain the following linear parabolic evolution equation:
\begin{equation*}
    \left( \partial_\tau - \Delta_{\omega_\psi} - \frac{X}{2} \right) \dot{\psi} = -\dot{\psi}.
\end{equation*}

From this, if $f$ denotes the normalized soliton potential with respect to K\"ahler-Ricci soliton $g$, we construct a crucial energy functional:
\begin{equation*}
    A_k(\tau) := \int_M \left(\dot{\psi}(\tau)\right)^{2k} e^{f + \frac{X}{2} \cdot \psi(\tau)} \omega_\psi(\tau)^n,
\end{equation*}
for an integer $ k$ sufficiently large. The central result is to show that $A_k(\tau) \equiv 0$, from which we conclude that the normalized Kähler-Ricci flow is a static flow.

The main technical challenge lies in establishing the \textbf{well-definedness} and \textbf{boundedness} of the energy $A_k(\tau)$.

Deruelle and Schulze \cite[Theorem 1.1]{Deruelle-Schulze} proved that for two AC soliton metrics $g_1$, $g_2$ whose soliton vector fields coincide outside of a compact set, then $\lim_{r\to+\infty}r^ne^{\frac{r^2}{4}}(g_1-g_2)$ exists and is in $L^2_{\operatorname{loc}}(C(S))$ topology, where $C(S)$ is the asymptotic Riemannian cone of real dimension $n$ with radial function $r$. This suggests that $|i\partial\bar{\partial}\psi|_{g_0}=O(e^{-f}f^{-n})$. To ensure $ A_k(\tau) $ is well-defined, we analyze the asymptotic behavior of $\dot{\psi}$ in Section \ref{A priori estimate at spatial infinity}. The evolution equation implies that
\begin{equation*}
    \dot{\psi} =O\left(e^{-f} f^{-n-1}\right)
\end{equation*}
at spatial infinity, with the decay exponent $n+1$ being nearly optimal. This decay ensures that, for sufficiently large $k$, the integrand in $A_k(\tau)$ decays rapidly enough to guarantee convergence outside of a sub-level set of the proper function $f$. Thus, the energy $A_k(\tau)$ is well-defined.

The remaining difficulty is to obtain uniform estimates for $\dot{\psi}$ and the volume form $\omega_\psi(\tau)^n$ on this sub-level set, where the decay at infinity no longer helps the analysis. To overcome this, we employ the maximum principle, using carefully chosen auxiliary functions and boundary data to control the behavior of $\dot{\psi}$ and the weighted volume ratio $e^{\frac{1}{2}X\cdot\psi}\frac{\omega_\psi(\tau)^n}{\omega^n}$ in this compact set, it turns out that the weighted volume ratio is a subsolution of the backward drift Laplacian $\Delta_{\omega_\psi}-\frac{X}{2}$. This allows us to extend the boundedness of the energy to the entire manifold, completing the proof that $A_k(\tau) \equiv 0$, and hence demonstrating the triviality of the normalized K\"ahler-Ricci flow.
\subsection{Outline of Paper} In Section \ref{Properties}, we will discuss some fundamental properties of the K\"ahler-Ricci soliton in Theorem \ref{main}, including soliton identities (Lemma \ref{soliton equalities}) and a lower bound for the scalar curvature (Lemma \ref{lower bound of scalar curvature}).  

In Section \ref{Ricci flow coming out of cone}, we will examine the general Ricci flow with a Riemannian conical metric as its initial data. Applying Perelman's pseudolocality theorem, we derive results that appeared in Siepmann’s PhD thesis \cite{Siepmann}, showing that the full Riemann curvature operator and its derivatives decay quadratically. As a consequence, we obtain a rough estimate for the Taylor expansion of such a Ricci flow (Proposition \ref{Taylor's expansion}).  

In Section \ref{Normalization of problem}, with the help of the Killing vector field $JX$, we reduce the K\"ahler-Ricci flow described in Theorem \ref{main} to a complex Monge-Amp\`ere equation for $\varphi$.  

In Section \ref{A priori estimate at spatial infinity}, we establish the polynomial decay at spatial infinity for several key quantities, utilizing our Taylor expansion formula (Proposition \ref{Taylor's expansion}). Furthermore, in subsection \ref{section max principle}, we apply the maximum principle to obtain exponential decay estimates.  

In Section \ref{soliton obstruction fl}, we use the flow generated by $ X $ to pull back the Ricci flow, leading to the \emph{normalized K\"ahler-Ricci flow} \eqref{obstruction flow}. By utilizing the Ricci curvature bound, we obtain time independent estimates in the compact set of $ M $ containing $ E $. Finally, we apply an energy method to show that the normalized K\"ahler-Ricci flow is a static flow, which ultimately implies the uniqueness theorem.
\subsection{Acknowledgments} I would like to express my heartfelt gratitude to my PhD supervisor, Alix Deruelle, for his patient guidance, continuous encouragement, and invaluable assistance with the key steps of the proof. I am also deeply thankful to my dear friend, Gérard Besson, whose inspiration and mentorship sparked and nurtured my interest in mathematical research. I would like to thank Professor Wenshuai Jiang and Professor Frank Pacard for their useful suggestions to the first version of this article.

\section{Properties of expanding gradient K\"ahler-Ricci soliton}\label{Properties}
Let $(M,g,X)$ be a gradient expanding Kähler–Ricci soliton as in Theorem \ref{main}. Denote by $\omega$ the Kähler form of $g$, and let $f\in C^\infty(M;\mathbb{R})$ be the Hamiltonian potential of $X$, characterized by the condition $\nabla^g f = X$. In this section, we provide a formal definition that captures the structure of the expanders that we study.
\begin{definition}[AC gradient K\"ahler-Ricci expander]\label{ACKR expander}
    An AC gradient K\"ahler-Ricci expander is a triple $(M,g,X)$ where $(M,g,X)$ is a complete expanding gradient K\"ahler-Ricci soliton of complex dimension $n\geq 2$ with complex structure $J$
whose curvature $\operatorname{Rm}(g)$ satisfies
\begin{equation*}
\sup_{x\in M}|(\nabla^{g})^{k}\operatorname{Rm}(g)|_{g}(x)d_{g}(p,\,x)^{2+k}<\infty\quad\textrm{for all $k\in\mathbb{N}_{0}$,}
\end{equation*}
where $d_{g}(p,\,\cdot)$ denotes the distance to a fixed point $p\in M$ with respect to $g$. 
\end{definition}
 Conlon, Deruelle and Sun \cite{Conlon-Deruelle-Sun} established a general structure theorem for an AC gradient K\"ahler-Ricci expander.

\begin{theorem}[Structure theorem \cite{Conlon-Deruelle-Sun}]\label{CDS theorem}Let $(M,g,X)$ be an AC gradient K\"ahler-Ricci expander. Then there exists a Kähler resolution $\pi: M \mapsto C$ of a Kähler cone $(C, g_0)$ with exceptional set $E$, such that
\begin{enumerate}
    \item the map $\pi$ is a biholomorphism between $M \setminus E$ and $C \setminus \{o\}$, where $o$ denotes the apex of the cone $C$;
    \item let $g(t)_{t>0}$ denote the self-similar solution to K\"ahler-Ricci flow associated to $g$, then when $t$ goes to 0, $\pi_*g(t)$ converges smoothly locally to $g_0$ on $C\setminus\{o\}$;
    \item let $r$ denote the radial function on the K\"ahler cone $(C, g_0)$, we have $d\pi(X)=r\partial_r$;
    \item the asymptotic conical gradient K\"ahler-Ricci expander $(M,g,X)$ is the unique (up to biholomorphisms) asymptotic conical gradient K\"ahler-Ricci expander with asymptotic cone $(C,g_0)$.
\end{enumerate}
\end{theorem}
A fundamental feature of gradient Kähler–Ricci expanders is encapsulated in the celebrated soliton identities, which play a central role in the analysis of their geometry and dynamics.
\begin{lemma}[Soliton identities]\label{soliton equalities}
    \begin{equation}
        \begin{split}
            &\Delta_\omega f=n+R_\omega,\\
            &\nabla^g R_\omega+\Ric(g)(X)=0,\\
            &|\partial f|_g^2+R_\omega=f+\textrm{constant}.
        \end{split}
    \end{equation}
    Here $n=\dim_\C M$, $\Delta_\omega$ is the K\"ahler Laplacian, $R_\omega=\frac{1}{2}R_g$ is the K\"ahler scalar curvature, $R_g$ is the scalar curvature of $g$.
\end{lemma}
\begin{proof}
    The proof is standard (see \cite[Section 2 of Chapter 1]{Chow}) since $\Ric(g)+g=\operatorname{Hess}_g f$.
\end{proof}
\textbf{From now on, we normalize $f$ such that $|\partial f|_g^2+R_\omega=f$}.
\begin{lemma}\label{f is proper}
    The normalized Hamiltonian potential $f$ is proper and satisfies $f\to\infty$ at infinity.
\end{lemma}
\begin{proof}
    See \cite[Proposition 2.19]{Conlon-Deruelle-Sun}.
\end{proof}
\begin{lemma}\label{lower bound of scalar curvature}
    There exists a constant $\varepsilon>0$ such that $R_\omega\ge-n+\varepsilon$ on $M$.
\end{lemma}
\begin{proof}
    We use $\Delta_g$ to denote the real Laplacian-Beltrami operator. By Bochner's formula
    \begin{equation*}
        \frac{1}{2}\Delta_g|\nabla f|_{g}^2=|\nabla^{g,2}f|^2_g+\Ric(g)(\nabla^g f,\nabla^g f)+g(\nabla^g\Delta_g f,\nabla^g f).
    \end{equation*}
    By Lemma \ref{soliton equalities}     \begin{equation*}
        \begin{split}
            &\frac{1}{2}\Delta_g|\nabla f|_{g}^2=\Delta_g|\partial f|^2_g=\Delta_g(f-R_\omega);\\
            &\Ric(g)(\nabla^g f,\nabla^g f)=-X\cdot R_\omega;\\
            &g(\nabla^g\Delta_g f,\nabla^g f)=2X\cdot R_\omega,
        \end{split}
    \end{equation*}
    and $\nabla^{g,2}f=g+\Ric(g)$, hence we have
    \begin{equation*}
        |\Ric(g)+g|^2_g+X\cdot R_\omega=\Delta_g(f-R_\omega)=2(n+R_\omega)-\Delta_g R_\omega.
    \end{equation*}
    Thus we have established an elliptic equation for $R_\omega+n$:
    \begin{equation*}
        \Delta_g(R_\omega+n)+X\cdot(R_\omega+n)=2(R_\omega+n)-|\Ric(g)+g|^2_g.
    \end{equation*}
    Since $R_\omega$ decays quadratically at infinity, the weak maximum principle implies that $R_\omega+n\ge0$ on $M$. Moreover, as $\Delta_g(R_\omega+n)+X\cdot(R_\omega+n)-2(R_\omega+n)\le 0$, Hopf's maximum principle \cite[Theorem 3.71]{T.Aubin} ensures that if there exists $x_0\in M$ such that $R_\omega(x_0)+n=0$, it immediately yields $R_\omega\equiv -n$, and therefore $\Ric(g)+g\equiv 0$. This gives us a contradiction since $|\Rm(g)|_g$ tends to $0$ at spatial infinity by assumption. Hence $R_\omega+n>0$ on $M$.

    As $R_\omega$ decays to 0 at spatial infinity, there exists a constant $\varepsilon>0$ such that $R_\omega\ge -n+\varepsilon$.
\end{proof}
\begin{corollary}\label{subharmonic function}
    The normalized Hamiltonian potential $f$ is a strictly subharmonic function on $M$.
\end{corollary}
\begin{proof}
    By Lemma \ref{soliton equalities} and \ref{lower bound of scalar curvature}, $\Delta_\omega f=R_\omega+n\ge\varepsilon>0$, which implies that $f$ is a strictly subharmonic function on $M$ as expected.
\end{proof}
\section{Ricci flows emerging from a Riemannian cone}\label{Ricci flow coming out of cone}
In this section, we discuss general properties of Ricci flows emerging from a Riemannian cone, including metric equivalence,  Taylor expansion formula and the corresponding estimates.

\begin{definition}\label{cone}
Let $(S, g_{S})$ be a closed Riemannian manifold. The \emph{Riemannian cone} $C(S)$ with \emph{link} $S$ is defined to be $\R_{>0}\times S$ with metric $g_C = dr^2 \oplus r^2g_{S}$ up to isometry. The radius function $r$ is then characterized
intrinsically as the distance from the apex in the metric completion.
\end{definition}
\begin{definition}[Ricci flows emerging from a Riemannian cone]\label{Def ricci flow coming out of cone}
   Assume that $M$ is an $n-$dimensional smooth noncompact manifold, and $(g(t))_{t\in (0,T)}$ is a family of smooth complete Riemannian metrics on $M$. We say that $g(t)$ is a Ricci flow emerging from the Riemannian cone $(C(S),g_C)$ if the following conditions hold:
   \begin{enumerate}
       \item the family $(g(t))_{t\in (0,T)}$ satisfies the Ricci flow equation: $\frac{\partial}{\partial t} g(t)=-2\Ric(g(t)),\quad t\in (0,T)$;
        \item there exists a diffeomorphism $\pi: M\setminus K\mapsto C(S)\setminus\{o\}$, where $K$ is a compact subset of $M$, $\{o\}$ is the vertex of $C(S)$;
        \item $\pi_*g(t)$ converges locally smoothly to $g_C$ away from the vertex $o$ as $t\to 0^+$;
        \item  for each compact interval $I\subset(0,T)$, the Riemannian curvature of $g(\cdot)$ is bounded on $M\times I$, i.e $\sup_{M\times I}|\Rm(g(\cdot))|_{g(\cdot)}<+\infty$.
   \end{enumerate}
\end{definition}
\begin{remark}
   By the assumptions (conical condition and curvature condition) in Theorem \ref{main}, the two K\"ahler-Ricci flows $g(t)_{t\in (0,+\infty)}$ and $g_\varphi(t)_{t\in (0,T)}$ in Theorem \ref{main} are, by definition, two Ricci flows emerging from the  K\"ahler cone $(C,g_0)$ as in Theorem \ref{main}. More generally, any Ricci flow with initial data given by a Riemannian cone and satisfying the Type~III$^\ast$ condition is a Ricci flow emerging from the Riemannian cone in the sense of Definition~\ref{Def ricci flow coming out of cone}.

\end{remark}
For convenience, we identify $M\setminus K $ with $C(S)\setminus\{o\}$ via the diffeomorphism $\pi$.
For a Ricci flow emerging from  the Riemannian cone, we have quadratic decay for the full Riemann curvature tensor and its derivatives.
\begin{theorem}\cite[Theorem 3.2.3]{Siepmann}\label{Siepmann}
  Let $(M^n,g(t))_{t\in(0,T)}$ be a complete Ricci flow emerging from a Riemannian cone $(C(S),g_C)$. Then $\Rm(g(t))$ decays quadratically, i.e., there exist constants $\{C_k>0\}_{k\in\N_0},\lambda>0$ independent of time such that
     \begin{equation}\label{quadratic decay}
         |(\nabla^{g(t)})^k\Rm(g(t))|_{g(t)}(x)\le \frac{C_k}{r(x)^{2+k}},\quad\textrm{for all $k\in\mathbb{N}_{0}$},
     \end{equation}
      for all $(x,t)\in M\times (0,T)$ such that $r(x)^2> \lambda t$. Throughout, we use $\mathbb{N}_0$ to denote the set of nonnegative integers.

\end{theorem}
\begin{proof}
    We only provide a sketch of proof here, for a detailed proof, see \cite[Theorem 3.2.3]{Siepmann}.

     First, since $(C(S),g_C)$ is smooth outside the vertex, there exists a uniform constant $\frac{1}{2}>\delta>0$ such that for all $x_0\in C(S)\setminus\{o\}$ we have: 
      \begin{equation*}
          |\Rm(g_C)|_{g_C}(x)\le \frac{1}{(2\delta r(x_0))^2},\quad \forall x\in B_{g_C}(x_0,2\delta r(x_0)),
      \end{equation*}
      and
      \begin{equation*}
          \operatorname{Vol}_{g_C}(B_{g_C}(x_0,2\delta r(x_0))\ge (1-\epsilon)\omega_n(2\delta r(x_0))^{n},\quad \forall x_0\in C(S)\setminus\{o\},
      \end{equation*}
 where $\omega_n$ denotes the volume of unit ball in $n$-dimensional Euclidean space and $\varepsilon$ denotes the universal constant given by Perelman's pseudolocality theorem (\cite[Theorem 10.1]{Perelman} and \cite[Theorem A.1]{Topping}).

 Fix $x_0\in C(S)\setminus\{o\}$. Since $g(t)$ converges to $g_C$ locally smoothly, hence there exists a $0<s_0<T$ (which potentially depends on the choice of $x_0$) such that for all $s\in [0,s_0]$
 \begin{equation*}
          |\Rm(g(s))|_{g(s)}(x)\le \frac{1}{(\delta r(x_0))^2},\quad \forall x\in B_{g(s)}(x_0,\delta r(x_0)),
      \end{equation*}
      and
      \begin{equation*}
          \operatorname{Vol}_{g(s)}\left(B_{g(s)}(x_0,\delta r(x_0))\right)\ge (1-\epsilon)\omega_n(\delta r(x_0))^{n}.
      \end{equation*}
      Now we fix $s\in [0,s_0]$, then by applying Perelman's pseudolocality theorem on $g(s+t)_{t\ge0}$, it follows that 
      \begin{equation*}
          |\Rm(g(s+t))|_{g(s+t)}(x)\le \frac{1}{(\delta\varepsilon r(x_0))^2}, 
      \end{equation*}
      for all $x\in B_{g(s+t)}(x_0,\delta\varepsilon r(x_0))$, $t\in [0,\min\{T-s,(\delta\varepsilon r(x_0))^2\})$. In particular
      \begin{equation*}
           |\Rm(g(s+t))|_{g(s+t)}(x_0)\le \frac{1}{(\delta\varepsilon r(x_0))^2}, 
      \end{equation*}
      for $t\in[0,\min\{T-s,(\delta\varepsilon r(x_0))^2\})$. Hence we conclude that for any $s\in [0,s_0]$
      \begin{equation*}
           |\Rm(g(t))|_{g(t)}(x_0)\le \frac{1}{(\delta\varepsilon r(x_0))^2}, 
      \end{equation*}
      for $(x_0,t)\in M\times(0,T)$ and $t<(\delta\varepsilon r(x_0))^2+s$. By letting $s\to 0$ and setting $\lambda:=(\delta\varepsilon)^{-2}$, we get results of \eqref{quadratic decay}. Estimates for higher covariant derivatives of the curvature
follow from \cite[Lemma A.4]{Topping}.
\end{proof}
\textbf{Define $\Omega_\lambda:=\{(x,t)\in M\setminus K\times (0,T)\ |\ r^2(x)>\lambda t \}$ as a parabolic set of space-time.} 
Given the above decay of the full Riemann curvature tensor,  the first observation is about equivalence of metrics.
\begin{lemma}\label{equivalence of metric} 
Let $(M^n,g(t))_{t\in(0,T)}$ be a complete Ricci flow emerging from a Riemannian cone $(C(S),g_C)$. There exist constants $\lambda>0$, $A_0>1$ independent of time such that for all $(x,t)\in\Omega_\lambda$
    \begin{equation}
        A_0^{-1}g_C(x)\le g(x,t)\le A_0g_C(x).
    \end{equation}
\end{lemma}
\begin{proof}
    For any $(x,t)\in\Omega_\lambda$, $(x,s)\in\Omega_\lambda$ for $0< s\le t$. Take a vector $V\in T_xM$ such that $g_C(x)(V,V)=1$, by \eqref{quadratic decay}
    \begin{equation}
        \begin{aligned}
            |\log g(x,t)(V,V)|&\le \int_0^t\frac{2\left|\Ric(g(x,s))(V,V)\right|}{g(x,s)(V,V)}ds\\
            &\le C(n)C_0\int_0^t\frac{1}{r(x)^2}ds=C(n)C_0\frac{t}{r(x)^2}\le C(n)C_0\frac{1}{\lambda}.\\
        \end{aligned}
    \end{equation}
    Take $A_0=\exp{\{C(n)C_0\frac{1}{\lambda}}\}$ and we have $A_0^{-1}g_C(x)\le g(x,t)\le A_0g_C(x)$ as required.
\end{proof}
The equivalence of metrics allow us to obtain a rough estimate of Taylor expansion of $g(t)$:
\begin{prop}[Taylor expansion formula]\label{Taylor's expansion}
Let $(M^n,g(t))_{t\in(0,T)}$ be a complete Ricci flow emerging from a Riemannian cone $(C(S),g_C)$. There exist constants $\{B_k>0\}_{k\in\N_0},\lambda>0$ independent of time such that for all $(x,t)\in\Omega_\lambda$
    \begin{equation}\label{Taylor expansion formula}
        \left|g(t)-\sum_{j=0}^k\frac{t^j}{j!}\left(\frac{\partial^j}{\partial t^j}g(t)\bigg|_{t=0}\right)\right|_{g_C}(x)\le B_{k+1}\frac{t^{k+1}}{r(x)^{2+2k}},\quad \textrm{for all $k\in\mathbb{N}_0$},
    \end{equation}
    \begin{equation}\label{gradient Taylor expansion}
        \left|\nabla^{g_C}\left(g(t)-\sum_{j=0}^k\frac{t^j}{j!}\bigg(\frac{\partial^j}{\partial t^j}g(t)\bigg|_{t=0}\bigg)\right)\right|_{g_C}(x)\le B_{k+1}\frac{t^{k+1}}{r(x)^{2+2k+1}}, \quad \textrm{for all $k\in\mathbb{N}_0$}.
    \end{equation}
\end{prop}
To prove this proposition, we need a lemma related to higher order time derivatives of metric along Ricci flow.
\begin{lemma}
    Let $(M,g(t))_{t\in (0,T)}$ be a smooth solution to Ricci flow. Then for $k\in\N^*$, the following equality holds pointwise:
        \begin{equation}\label{time derivative along RF}
            \frac{\partial^k}{\partial t^k}g(t)=\sum_{p+q=0}^k\sum_{\sum_{j=1}^p i_j=2q}(\nabla^{g(t)})^{i_1}\Rm(g(t))*\cdot\cdot\cdot*(\nabla^{g(t)})^{i_p}\Rm(g(t))*\Rm(g(t))^{k-p-q}.
        \end{equation}
        Here we use the notation $T(t)*S(t)$ for the linear combination of contractions of tensors with respect to metric $g(t)$, and $\underbrace{T*\cdot\cdot\cdot*T}_{n}:=T^n$ for tensor $T$. Throughout, we use $\mathbb{N}^\ast$ to denote the set of positive integers.
\end{lemma}
\begin{proof}
           We argue by induction using the evolution equation of $\nabla^{g(t),k}\Rm$. For convenience, we omit all dependence related to $ g(t)$.

When $k=1$, \eqref{time derivative along RF} holds automatically by the Ricci flow equation. Assume that \eqref{time derivative along RF} is true for $k\ge1$.
     \begin{equation*}
         \frac{\partial^{k+1}}{\partial t^{k+1}}g(t)=\frac{\partial}{\partial t}\left(\sum_{p+q=0}^k\sum_{\sum_{j=1}^p i_j=2q}\nabla^{i_1}\Rm*\cdot\cdot\cdot*\nabla^{i_p}\Rm*\Rm^{k-p-q}\right).
     \end{equation*}
Recall that (see \cite[Theorem 13.2]{R.Hamilton})
\begin{equation}\label{eq: evolution equation of Rm}
    \frac{\partial}{\partial t}(\nabla^m\Rm)=\Delta(\nabla^m\Rm)+\sum_{l=0}^m\nabla^l\Rm*\nabla^{m-l}\Rm.
\end{equation}
     Hence for any $1\le j\le p$
\begin{equation*}
     \begin{aligned}
        &\quad\nabla^{i_1}\Rm*\cdot\cdot\cdot\frac{\partial}{\partial t}\nabla^{i_j}\Rm *\cdot\cdot\cdot*\nabla^{i_p}\Rm*\Rm^{k-p-q}\\ 
    &=\nabla^{i_1}\Rm*\cdot\cdot\cdot(\Delta(\nabla^{i_j}\Rm)+\sum_{l=0}^{i_j}\nabla^l\Rm*\nabla^{i_j-l}\Rm) *\cdot\cdot\cdot*\nabla^{i_p}\Rm*\Rm^{k-p-q}\\
        &=\sum_{l=0}^{i_j}\nabla^{i_1}\Rm*\cdot\cdot\cdot\nabla^l\Rm*\nabla^{i_j-l}\Rm *\cdot\cdot\cdot*\nabla^{i_p}\Rm*\Rm^{k-p-q}\\
        &\quad +\nabla^{i_1}\Rm*\cdot\cdot\cdot*\nabla^{i_j+2}\Rm *\cdot\cdot\cdot*\nabla^{i_p}\Rm*\Rm^{k-p-q}.\\
    \end{aligned}
    \end{equation*}
  If $k-p-q\neq0$
    \begin{equation*}
     \begin{aligned}
        &\quad \nabla^{i_1}\Rm*\cdot\cdot\cdot*\nabla^{i_p}\Rm*\frac{\partial}{\partial t}(\Rm^{k-p-q}) \\
        &=\nabla^{i_1}\Rm*\cdot\cdot\cdot*\nabla^{i_p}\Rm*\Rm^{k-p-q-1}*(\Delta\Rm+\Rm*\Rm)\\
        &=\nabla^{i_1}\Rm*\cdot\cdot\cdot*\nabla^{i_p}\Rm*\Rm^{k+1-p-q}+\nabla^{i_1}\Rm*\cdot\cdot\cdot*\nabla^{i_p}\Rm*\nabla^2\Rm*\Rm^{k-p-q-1}.\\
    \end{aligned}
    \end{equation*}
   All of these terms above belong to $\sum_{p+q=0}^{k+1}\sum_{\sum_{j=1}^p i_j=2q}\nabla^{i_1}\Rm*\cdot\cdot\cdot*\nabla^{i_p}\Rm*\Rm^{k+1-p-q}$, thus \eqref{time derivative along RF} holds for $k+1$.
   
   By induction, \eqref{time derivative along RF} holds for all $k\in\N_0$.
\end{proof}
Using \eqref{time derivative along RF}, we now prove Lemma \ref{Taylor's expansion} directly. 
\begin{proof}[Proof of Proposition \ref{Taylor's expansion}]
In this proof $\{B_k>0\}_{k\in\N_0}$ denotes constants independent of time that may vary from line to line.

  Fix $k\in\N_0$, performing a Taylor expansion with integral remainder,
           \begin{equation*}
               \begin{aligned}
                     \left|g(t)-\sum_{j=0}^k\frac{t^j}{j!}\left(\frac{\partial^j}{\partial t^j}g(t)\bigg|_{t=0}\right)\right|_{g_C}(x)&\le \frac{1}{k!}\int_0^t\left|\frac{\partial^{k+1}}{\partial s^{k+1}}g(s)\right|_{g_C}(x)(t-s)^kds\\
                &\le C(n) \frac{A_0}{k!}\int_0^t \left|\frac{\partial^{k+1}}{\partial s^{k+1}}g(s)\right|_{g(s)}(x)(t-s)^kds.\\
               \end{aligned}
           \end{equation*}
           By \eqref{quadratic decay} and \eqref{time derivative along RF}, there exists a constant $B_{k+1}>0$ such that
           \begin{equation*}
                \begin{aligned}
                \left|\frac{\partial^{k+1}}{\partial s^{k+1}}g(s)\right|_{g(s)}(x)&\le C(n)\sum_{p+q=0}^{k+1}\sum_{\sum_{j=1}^p i_j=2q}|\nabla^{i_1}\Rm|\cdot\cdot\cdot|\nabla^{i_p}\Rm||\Rm|^{k+1-p-q}(x)\\
                & \le B_{k+1}\frac{1}{r(x)^{2+2k}}.
               \end{aligned}
           \end{equation*}
           Hence \eqref{Taylor expansion formula} holds for all $k\in \N_0$.

           To obtain \eqref{gradient Taylor expansion}, we first need to estimate $h(x,t):=\nabla^{g_C}(g(t)-g_C)(x)$ for $(x,t)\in\Omega_\lambda$. Fix $(x,t)\in \Omega_\lambda$ and define $y(s):=|h(x,s)|_{g_C}^2$ for $0< s\le t$. Since $(x,s)\in\Omega_\lambda$,
           \begin{equation*}
               \begin{aligned}
                   \frac{d}{ds}y(s)&\le 4|\nabla^{g_C}\Ric{(g(s))}|_{g_C}(x)\sqrt{y(s)}\\
        &\le4\left(|(\nabla^{g_C}-\nabla^{g(s)})\Ric(g(s))|_{g_C}(x)+|\nabla^{g(s)}\Ric(g(s))|_{g_C}(x)\right)\sqrt{y(s)}.\\
               \end{aligned}
           \end{equation*}
           Notice that
    \begin{equation*}
        \begin{aligned}
          (\nabla^{g_C}-\nabla^{g(s)})\Ric(g(s))&=g(s)^{-1}*h(x,s)*\Ric(g(s)),\\
            |(\nabla^{g_C}-\nabla^{g(s)})\Ric(g(s))|_{g_C}&\le C(n)|g(s)^{-1}|_{g_C}(x)|\Ric(g(s))|_{g_C}(x)\sqrt{y(s)}\\
            &\le C(n)A_0|\Ric(g(s))|_{g_C}(x)\sqrt{y(s)}\\
            &\le C(n)A_0^2|\Ric(g(s))|_{g(s)}(x)\sqrt{y(s)}\\
            &\le C(n)C_0A_0^2\frac{1}{r(x)^2}\sqrt{y(s)},\\
        \end{aligned}
    \end{equation*}
    and by \eqref{quadratic decay},
    \begin{equation*}
        |\nabla^{g(s)}\Ric(g(s))|_{g_C}(x)\le C(n)A_0^{\frac{3}{2}} |\nabla^{g(s)}\Ric(g(s))|_{g(s)}(x)\le C(n)A_0^{\frac{3}{2}}C_1\frac{1}{r(x)^3}.
    \end{equation*}
    Hence by Young's inequality
    \begin{equation*}
         \frac{d}{ds}y(s)\le \left(C(n)A_0^{\frac{3}{2}}C_1\frac{1}{r(x)^3}+C(n)C_0A_0^2\frac{1}{r(x)^2}\sqrt{y(s)}\right)\sqrt{y(s)}\le B_0\left(\frac{y(s)}{r(x)^2}+\frac{1}{r(x)^4}\right).
    \end{equation*}
    Integrating and using the initial
condition $y(0)=0$, we have 
    \begin{equation}\label{eq: ODE integration}
        y(s)\le \frac{1}{r(x)^2}(e^{B_0\frac{s}{r(x)^2}}-1)\le \frac{1}{r(x)^2}e^{B_0\frac{s}{r(x)^2}}\le  B_0\frac{1}{r(x)^2},
    \end{equation}
    as $s\le t\le \lambda
    ^{-1}r(x)^2$.

    Moreover,
    \begin{equation*}
        \begin{aligned}
             |h(x,t)|_{g_C}&\le 2\int_0^t|\nabla^{g_C}\Ric(g(s))|_{g_C}(x)ds\\
        &\le2 \int_0^t \left(|(\nabla^{g_C}-\nabla^{g(s)})\Ric(g(s))|_{g_C}(x)+|\nabla^{g(s)}\Ric(g(s))|_{g_C}(x)\right)ds\\
        &\le B_0\int_0^t \left(\frac{1}{r(x)^2}\sqrt{y(s)}+\frac{1}{r(x)^3}\right) ds\\
        &\le B_0\int_0^t \frac{1}{r(x)^3}ds=B_0\frac{t}{r(x)^3}.\\
        \end{aligned}
    \end{equation*}
    For higher order expansion, performing a Taylor expansion with integral remainder:
    \begin{equation*}
       \left|\nabla^{g_C}\left(g(t)-\sum_{j=0}^k\frac{t^j}{j!}\bigg(\frac{\partial^j}{\partial t^j}g(t)\bigg|_{t=0}\bigg)\right)\right|_{g_C}(x)\le \frac{1}{k!}\int_0^t\left|\nabla^{g_C}\left(\frac{\partial^{k+1}}{\partial s^{k+1}}g(s)\right)\right|_{g_C}(x)(t-s)^kds.
    \end{equation*}
    Notice that,
    \begin{equation*}
        \begin{aligned}
            \nabla^{g_C}\left(\frac{\partial^{k+1}}{\partial s^{k+1}}g(s)\right)(x)&=(\nabla^{g_C}-\nabla^{g(s)})\left(\frac{\partial^{k+1}}{\partial s^{k+1}}g(s)\right)(x)+\nabla^{g(s)}\left(\frac{\partial^{k+1}}{\partial s^{k+1}}g(s)\right)(x)\\
                &=g(s)^{-1}*h(x,s)*\left(\frac{\partial^{k+1}}{\partial s^{k+1}}g(s)\right)(x)+\nabla^{g(s)}\left(\frac{\partial^{k+1}}{\partial s^{k+1}}g(s)\right)(x).
        \end{aligned}
    \end{equation*}
   Hence,
     \begin{equation*}
         \left|\nabla^{g_C}\left(\frac{\partial^{k+1}}{\partial s^{k+1}}g(s)\right)\right|_{g_C}(x)\le C(n)A_0^2|h(x,s)|_{g_C}\left|\frac{\partial^{k+1}}{\partial s^{k+1}}g(s)\right|_{g_C}(x)+\left|\nabla^{g(s)}\left(\frac{\partial^{k+1}}{\partial s^{k+1}}g(s)\right)\right|_{g_C}(x),
     \end{equation*}
     for each term by \eqref{quadratic decay} and \eqref{time derivative along RF}:
            \begin{equation*}
                \left|\frac{\partial^{k+1}}{\partial s^{k+1}}g(s)\right|_{g_C}(x)\le C(n)A_0\left|\frac{\partial^{k+1}}{\partial s^{k+1}}g(s)\right|_{g(s)}(x)\le C(n)A_0B_{k+1}\frac{1}{r(x)^{2+2k}}=B_{k+1}\frac{1}{r(x)^{2+2k}},
            \end{equation*}
            \begin{equation*}
                \left|\nabla^{g(s)}\left(\frac{\partial^{k+1}}{\partial s^{k+1}}g(s)\right)\right|_{g_C}(x)\le C(n)A_0^{\frac{3}{2}}\left|\nabla^{g(s)}\left(\frac{\partial^{k+1}}{\partial s^{k+1}}g(s)\right)\right|_{g(s)}(x)\le B_{k+1}\frac{1}{r(x)^{2+2k+1}}.
            \end{equation*}
            Thus,
            \begin{equation*}
                \left|\nabla^{g_C}\left(\frac{\partial^{k+1}}{\partial s^{k+1}}g(s)\right)\right|_{g_C}(x)\le B_{k+1}\frac{1}{r(x)^{2+2k+1}}.
            \end{equation*}
            As a consequence, we have
            \begin{equation*}
                \left|\nabla^{g_C}\left(g(t)-\sum_{j=0}^k\frac{t^j}{j!}\bigg(\frac{\partial^j}{\partial t^j}g(t)\bigg|_{t=0}\bigg)\right)\right|_{g_C}(x)\le \int_0^t B_{k+1}\frac{1}{r(x)^{2+2k+1}}(t-s)^kds\le B_{k+1}\frac{t^{k+1}}{r(x)^{2+2k+1}}.
            \end{equation*}
     Thus \eqref{gradient Taylor expansion} holds for all $k\in\N_0$.
\end{proof}
\section{Setup of a complex Monge-Amp\`ere equation}\label{Normalization of problem}

Let $(M,g,J,X)$ be a gradient expanding K\"ahler-Ricci soliton as in Theorem \ref{main}, $(C,g_0)$ be the K\"ahler cone as in Theorem \ref{main}, and let $g(t)$ be the forward self-similar K\"ahler-Ricci flow associated to $g$. Suppose that $f\in C^\infty(M;\R)$ is the normalized Hamiltonian potential of vector field $X$ such that $\nabla^gf=X$ and $|\partial f|^2_g+R_\omega=f$.
\subsection{Symmetries of the K\"ahler potential}
\begin{lemma}\label{KV}
 The Reeb vector field $JX$ is a Killing vector field for $g(t)$ for all $t\in (0,\infty)$.
\end{lemma}
\begin{proof}
    The Reeb vector field $JX$ is a Killing vector field for $g$. 
    
    For any two vector fields $Y,Z$, we have
    \begin{align*}
        (\mathcal{L}_{JX}g)(Y,Z)&=g(\nabla^g_YJX,Z)+g(Y,\nabla^g_ZJX)\\
        &=g(\nabla^g_{JY}X,Z)-g(JY,\nabla^g_ZX)\\
        &=\nabla^{g,2}f(JY,Z)-\nabla^{g,2}f(Z,JY)\\
        &=0.
    \end{align*}
    As $X$ is a real-holomorphic vector field, $[JX,X]=-J([X,X])=0$, this implies that $\Phi_t^*(JX)=JX$ for any $t\in(0,\infty)$. Since $g(t)=t\Phi_t^*g$, and $\mathcal{L}_{JX}g=0$, we have $\mathcal{L}_{JX}g(t)=0$ as expected.
\end{proof}
In Theorem \ref{main}, we assume that $JX$ is Killing for $g_\varphi(t_0)$ for some $t_0\in (0,T)$. The following lemma shows $JX$ is Killing for $g_\varphi(t)$ for all $t\in (0,T).$
\begin{lemma}\label{Killing condition}
    Let $g_\varphi(t)_{t\in (0,T)}$ be the solution to the K\"ahler-Ricci flow in Theorem \ref{main}, then $\mathcal{L}_{JX}g_\varphi(t)=0$ for all $t\in (0,T)$.
\end{lemma}
\begin{proof}
    From the Killing condition of Theorem \ref{main}, there exists $t_0\in (0,T)$ such that $\mathcal{L}_{JX}g_\varphi(t_0)=0$. 
    Fix $t_0<t_1<T$, by the curvature hypothesis as in Theorem \ref{main}, there exists a constant $K_1>0$ such that for all $t\in [t_0,t_1]$,
    \begin{equation*}
        \sup_M|\Rm(g_\varphi(t))|_{g_\varphi(t)}\le K_1.
    \end{equation*}
    By the uniqueness theorem of \cite[Theorem 1.1]{Chen-Zhu}, since $JX$ is a Killing vector field for $g_\varphi(t_0)$, $JX$ is a Killing vector field for $g_\varphi(t_1)$.

    Now fix $0<t_2<t_0$; by the curvature condition as in Theorem \ref{main}, there exists a constant $K_2>0$ such that for all $t\in [t_2,t_0]$,
    \begin{equation*}
        \sup_M|\Rm(g_\varphi(t))|_{g_\varphi(t)}\le K_2.
    \end{equation*}
    As a consequence of the backward uniqueness theorem due to \cite[Theorem 1.1]{Kotschwar3}, since $JX$ is a Killing vector field for $g_\varphi(t_0)$, $JX$ is a Killing vector field for $g_\varphi(t_2)$.
    
    Hence $JX$ is a Killing vector field for $g_\varphi(t)$ for all $t\in (0,T)$.
\end{proof}
We have shown that $JX$ is a Killing vector field for $g_\varphi(t)$ for all $t\in (0,T)$. The following lemma implies that the K\"ahler potential $\varphi$ can be chosen to have
additional symmetry, thanks to the Killing vector field $JX$.
\begin{lemma}\label{JX varphi=0}
    Let $g_\varphi(t)$ be the solution to K\"ahler-Ricci flow in Theorem \ref{main}. There exists a smooth real-valued function $\tilde{\varphi}\in C^\infty(M\times (0,T))$ such that $JX\cdot\tilde{\varphi}=0$ and $i\partial\bar{\partial}\tilde{\varphi}=i\partial\bar{\partial}\varphi$.
\end{lemma}
\begin{proof}
    Consider the isometry group of $(M,\,g)$ that fixes $E$ and is endowed with the topology induced by uniform convergence on compact
subsets of $M$. By the Arzel\`a-Ascoli theorem, this is a compact Lie group. \cite{Conlon-Deruelle-Sun} also proves that the flows of $X$ and $JX$ preserve $E$ (see \cite[Lemma 2.6]{Conlon-Deruelle-Sun}). Thus, in particular, the flow of $JX$ lies in this isometry group, taking the closure of the flow of $JX$ in this group therefore yields the holomorphic isometric action of a torus $T$ on $(M,\,J,\,g)$.

Since $JX$ is Killing with respect to $g_\varphi$ by Lemma \ref{Killing condition}, all elements in $T$ fix $g_\varphi$. Let $\mu_T$ be the normalized Haar measure of $T$, and we define
\begin{equation}\label{average of varphi}
    \Tilde{\varphi}:=\int_{a\in T}a^*\varphi d\mu_T(a).
\end{equation}
By construction, $\Tilde{\varphi}$ is invariant under the flow of $JX$, so it yields that $JX\cdot\Tilde{\varphi}=0$. In addition, Lemma \ref{KV} and Lemma \ref{Killing condition} imply that $0=\mathcal{L}_{JX}(\omega_\varphi)=\mathcal{L}_{JX}(i\partial\bar{\partial}\varphi)=i\partial\bar{\partial}(JX\cdot\varphi)$. By \eqref{average of varphi} we have 
\begin{equation*}
    i\partial\bar{\partial}\Tilde{\varphi}=\int_{a\in T}a^*(i\partial\bar{\partial}\varphi) d\mu_T(a)=i\partial\bar{\partial}\varphi,
\end{equation*}
as expected.
\end{proof}
Since for all $t\in (0,T)$ $\omega_\varphi(t)=\omega(t)+i\partial\bar{\partial}\varphi(t)=\omega(t)+i\partial\bar{\partial}\Tilde{\varphi}(t)$. From now on, we choose to use $g_{\tilde{\varphi}}(\omega_{\tilde{\varphi}})$ instead of $g_\varphi(\omega_\varphi)$. For convenience, we keep the notation $\varphi$ for the K\"ahler potential.
\subsection{Reduction to a complex Monge-Amp\`ere equation}
The symmetries of the K\"ahler potential help us to reduce the K\"ahler-Ricci flow equation to a complex Monge-Amp\`ere equation.
\begin{prop}[Complex Monge-Amp\`ere equation]\label{Proposition monge ampere}
    Let $g_\varphi(t)_{t\in (0,T)}$ be the solution to K\"ahler-Ricci flow in Theorem \ref{main} satisfying
    \begin{equation}\label{K-R varphi flow}
        \begin{split}
            \frac{\partial}{\partial t}g_\varphi(t)&=-\Ric(g_\varphi(t)),\\
            JX\cdot\varphi(t)&=0 \quad t\in (0,T).
        \end{split}
    \end{equation}
    Then there exists a smooth real-valued function $\bar{\varphi}\in C^\infty(M\times (0,T))$ such that $i\partial\bar{\partial}\bar{\varphi}=i\partial\bar{\partial}\varphi$ and
    \begin{equation}\label{Monge-Ampere equation}
         \begin{split}
             \frac{\partial}{\partial t}\bar{\varphi}(t)&=\log\frac{\omega_{\bar{\varphi}}(t)^n}{\omega(t)^n},\\
             JX\cdot\bar\varphi(t)&=0\quad t\in (0,T),
         \end{split}
    \end{equation}
    such that $\pi_*\bar{\varphi}(t)$ converges to $0$ locally smoothly when $t$ tends to $0$. Here $\omega_{\bar{\varphi}}(t)=\omega(t)+i\partial\bar{\partial}\bar{\varphi}(t)$ for all $t\in (0,T).$
\end{prop}
\begin{proof}

    Subtracting the two K\"ahler-Ricci flow equations gives
    \begin{equation*}
        i\partial\bar{\partial}\frac{\partial}{\partial t}\varphi(t)= \rho_{\omega(t)}-\rho_{\omega_\varphi(t)}.
    \end{equation*}
    On the K\"ahler manifold $M$
    \begin{equation*}
        \rho_{\omega(t)}-\rho_{\omega_\varphi(t)}=i\partial\bar{\partial}\log\frac{\omega_\varphi(t)^n}{\omega(t)^n}.
    \end{equation*}
    This yields that 
    \begin{equation}
         i\partial\bar{\partial}\left(\frac{\partial}{\partial t}\varphi(t)-\log\frac{\omega_\varphi(t)^n}{\omega(t)^n}\right)=0.
    \end{equation}
    Let $\kappa(t):=\frac{\partial}{\partial t}\varphi(t)-\log\frac{\omega_\varphi(t)^n}{\omega(t)^n}$. We now show that $\kappa(t)$ is indeed spatially constant for each fixed $t$. The argument is originally due to Siepmann in his PhD thesis \cite{Siepmann}; see also \cite[Lemma~2.4]{Conlon-Deruelle2}.

    We have $i\partial\bar{\partial}\kappa(t)=0$ and $JX\cdot \kappa(t)=0$ thanks to Lemma \ref{Killing condition} and Lemma \ref{JX varphi=0}. If we denote $X^{1,0}$ to be the holomorphic part of $X$, then $JX\cdot\kappa(t)=0$ is equivalent to saying that $X^{1,0}\cdot\kappa(t)$ is real. Since $i\partial\bar{\partial}\kappa(t)=0$, it follows that
    $\bar{\partial}(X^{1,0}\cdot\kappa(t))=0$. This means $X^{1,0}\cdot\kappa(t)$ is a real-valued holomorphic function on $M$. This implies that $X^{1,0}\cdot\kappa(t)$ is spatially constant for each fixed $t$.
 Therefore $X\cdot\kappa(t)$ is a constant denoted by $c$.

    Since $d\pi(X)=r\partial_r$, equation $r\partial_r\pi_*\kappa(t)=c$ means $\pi_*\kappa(t)=c\log r+\bar{\kappa}(t)$, where $\bar{\kappa}$ denotes a smooth function that is only defined on the link $S$ of $C$. Moreover, as $i\partial\bar{\partial}(\pi_*\kappa(t))=0$, in particular $\Delta_{g_0}\pi_*\kappa(t)=0$, $n\ge2$, so that,
    \begin{equation*}
        0=\Delta_{g_0}\pi_*\kappa(t)=cr^{-2}(2n-2)+r^{-2}\Delta_{g_S}\bar{\kappa}.
    \end{equation*}
    This implies that $c=0$ and $\bar{\kappa}$ is a constant. Since $M\setminus E$ is dense on $M$, $\kappa(t)$ is a constant in space denoted by $c(t).$ Hence
    \begin{equation}\label{Monge-Ampere equation tilde}
        \frac{\partial}{\partial t}\varphi(t)-\log\frac{\omega_\varphi(t)^n}{\omega(t)^n}=c(t).
    \end{equation}
    Take $C(t)$ as a primitive of $-c(t)$ such that $C'(t)=-c(t)$, and let $\tilde{\varphi}(t)=\varphi(t)+C(t)$, we have
    \begin{equation*}
        \frac{\partial}{\partial t}\tilde{\varphi}(t)-\log\frac{\omega_{\tilde{\varphi}}(t)^n}{\omega(t)^n}=0.
    \end{equation*}

 By the conical condition stated in Theorem \ref{main}, one can extend $\pi_*\tilde{\varphi}(t)$ to a smooth function at time 0 outside the vertex, which we denote by $\varphi_0$.

In particular, $i\partial\bar{\partial}\varphi_0=0$ in $C\setminus\{o\}$ and $J_0(r\partial_r)\cdot \varphi_0=0$. By the same reason as above $\varphi_0$ is a constant $b$. Let $\bar\varphi(t)=\tilde\varphi(t)-b$, we have $i\partial\bar\partial\bar\varphi=i\partial\bar\partial\varphi$, $JX\cdot\bar\varphi=JX\cdot\varphi=0$ and 
 \begin{equation*}
         \frac{\partial}{\partial t}\bar{\varphi}(t)=\log\frac{\omega_{\bar{\varphi}}(t)^n}{\omega(t)^n},\quad t\in (0,T),
    \end{equation*}
    where $\pi_*\bar\varphi$ converges to 0 when $t$ tends to 0 as desired.
\end{proof}
\textbf{Since $i\partial\bar\partial\bar\varphi=i\partial\bar\partial\varphi$ and $JX\cdot\bar\varphi=JX\cdot\varphi=0$, from now on, we choose to use $g_{\bar\varphi}$(resp. $\omega_{\bar\varphi}$) instead of $g_\varphi$(resp. $\omega_\varphi$). For convenience, we keep the notation $\varphi$ for the K\"ahler potential.}
\section{A priori estimates at spatial infinity}\label{A priori estimate at spatial infinity}
We adapt the notation $\Omega_\lambda$ mentioned in Section \ref{Ricci flow coming out of cone}, and we identify $M\setminus E$ with $C\setminus\{o\}$ as in Theorem \ref{main} through the biholomorphism $\pi$. Recall that $\Omega_\lambda:=\{(x,t)\in M\setminus E\times (0,T)\ |\ r^2(x)>\lambda t \}$ where $\lambda>0$ is the constant appearing in Theorem~\ref{Siepmann}.

In this section, we aim to derive precise estimates in this parabolic set $\Omega_\lambda$. Building on the expansion formula developed (Proposition \ref{Taylor's expansion}), we first establish polynomial decay in $\Omega_\lambda$. Then, by applying the maximum principle, we obtain a nearly optimal estimate, which ultimately leads to exponential decay.

Recall the setting of $\varphi$ as in Proposition \ref{Proposition monge ampere}:
\begin{equation*}
    \begin{split}
        &\frac{\partial}{\partial t}\varphi(t)=\log\frac{\omega_\varphi(t)^n}{\omega(t)^n},\\
        & JX\cdot\varphi(t)=0,\quad \forall t\in (0,T),
    \end{split}
\end{equation*}
with $\pi_*\varphi(t)$ converging to $0$ locally smoothly when $t$ tends to $0$.
\subsection{Polynomial decay}
From Section \ref{Ricci flow coming out of cone}, both $g(t)$ and $g_\varphi(t)$ as in Theorem \ref{main} are Ricci flows emerging from  the same Kähler cone $(C,g_0)$ and share the same Taylor expansion formula. 

As a consequence of Proposition \ref{Taylor's expansion} applied to $g_\varphi(t)_{t\in (0,T)}$ and $g(t)_{t\in (0,\infty)}$, we have the following proposition:
\begin{prop}\label{consequence of taylor expansion}
    There exist constants $\{B_k>0\}_{k\in\N_0}$ and $\lambda>0$ independent of time such that
    \begin{equation}\label{polynomial estimate of Hessian and gradient of Hessian}
        \begin{split}
             & |i\partial\bar{\partial}\varphi(x,t)|_{g_0}\le B_k\frac{t^k}{r(x)^{2k}},\quad\textrm{for all $(x,t)\in\Omega_\lambda$, $k\in\N_0$},\\
             & |\nabla^{g_0}(i\partial\bar{\partial}\varphi(x,t))|_{g_0}\le B_k\frac{t^k}{r(x)^{2k+1}},\quad\textrm{for all $(x,t)\in\Omega_\lambda$, $k\in\N_0$}.\\
        \end{split}
    \end{equation}
\end{prop}
Thanks to the complex Monge-Amp\`ere equation \eqref{Monge-Ampere equation}, we establish the polynomial decay of $\varphi,X\cdot\varphi$ and $\frac{\partial}{\partial t}\varphi:=\dot{\varphi}$.
\begin{prop}[Polynomial decay]
    For $\lambda>0$ as in Proposition \ref{polynomial estimate of Hessian and gradient of Hessian}, for any $k\in \N_0$, there exists a constant $C_k>0$ independent of time, such that for all $(x,t)\in\Omega_\lambda$:
    \begin{equation}\label{polynomial decay formula}
        \begin{split}
            &|\dot{\varphi}(x,t)|\le C_k\frac{t^k}{r(x)^{2k}},\\
              &|\varphi(x,t)|\le C_k\frac{t^{k+1}}{r(x)^{2k}},\\
            &|X\cdot\dot\varphi(x,t)|\le C_k\frac{t^k}{r(x)^{2k}},\\
            &|X\cdot\varphi(x,t)|\le C_k\frac{t^{k+1}}{r(x)^{2k}}.
        \end{split}
    \end{equation}
\end{prop}
\begin{proof}
For $(x,t)\in\Omega_\lambda$,
    the complex Monge-Amp\`ere equation \eqref{Monge-Ampere equation} implies that
    \begin{equation*}
        \begin{split}
            &\dot{\varphi}(x,t)\le\Delta_{\omega(t)}\varphi(x,t),\\
            &\dot{\varphi}(x,t)\ge \Delta_{\omega_\varphi(t)}\varphi(x,t).\\
        \end{split}
    \end{equation*}
    Since $|\Delta_{\omega(t)}\varphi(x,t)|\le C(n)|g(t)^{-1}|_{g_0}|i\partial\bar{\partial}\varphi(x,t)|_{g_0}\le C(n)A_0B_k\frac{t^k}{r(x)^{2k}}$, take $C_k:=C(n)A_0B_k$ such that $\dot{\varphi}(x,t)\le C_k\frac{t^k}{r(x)^{2k}}$. From now on, we use $C_k>0$ to denote a constant independent of time which would vary from line to line. 

    Moreover, $|\Delta_{\omega_\varphi(t)}\varphi(x,t)|\le C(n)|g_\varphi(t)^{-1}|_{g_0}|i\partial\bar{\partial}\varphi(x,t)|_{g_0}\le C_k\frac{t^k}{r(x)^{2k}}$, hence $|\dot{\varphi}(x,t)|\le C_k\frac{t^k}{r(x)^{2k}}$. By integration and $\lim_{t\to0^+}\varphi(x,t)=0$, hence $|\varphi(x,t)|\le C_k\frac{t^{k+1}}{r(x)^{2k}}$ follows immediately.

Recall that $X\cdot\dot\varphi(t)=X\cdot\log\frac{\omega_\varphi(t)^n}{\omega(t)^n}=\tr_{\omega_\varphi(t)}(\mathcal{L}_X\omega_\varphi(t))-\tr_{\omega(t)}(\mathcal{L}_X\omega(t))$.

On the one hand \begin{equation*}
 |\tr_{\omega_\varphi(t)}(\mathcal{L}_X\omega_\varphi(t))-\tr_{\omega(t)}(\mathcal{L}_X\omega(t))|\le   |\tr_{\omega_\varphi(t)}(\mathcal{L}_X\omega(t))-\tr_{\omega(t)}(\mathcal{L}_X\omega(t))|_{g_0}+|\tr_{\omega_\varphi(t)}\left(\mathcal{L}_X(i\partial\bar{\partial}\varphi(t))\right)|_{g_0}.
\end{equation*}
On the other hand
\begin{equation*}
    \begin{split}
        \tr_{\omega_\varphi(t)}(\mathcal{L}_X\omega(t))-\tr_{\omega(t)}(\mathcal{L}_X\omega(t))&=\nabla^{g(t)}X*(g_\varphi(t)^{-1}-g(t)^{-1})=\nabla^{g(t)}X*\partial\bar{\partial}\varphi(t)*g_\varphi(t)^{-1}*g(t)^{-1},\\
        \tr_{\omega_\varphi(t)}\left(\mathcal{L}_X(i\partial\bar{\partial}\varphi(t))\right)&=g_\varphi(t)^{-1}*\nabla^{g_0}X*i\partial\bar\partial\varphi(t)+g_\varphi(t)^{-1}*X*\nabla^{g_0}i\partial\bar\partial\varphi(t).
    \end{split}
\end{equation*}
For $(x,t)\in\Omega_\lambda$, by Proposition \ref{polynomial estimate of Hessian and gradient of Hessian} we have
\begin{equation*}
    \begin{split}
        &|\tr_{\omega_\varphi(t)}(\mathcal{L}_X\omega(t))-\tr_{\omega(t)}(\mathcal{L}_X\omega(t))|_{g_0}(x)\\
        \le &C(n)|\nabla^{g(t)}X|_{g_0}|g_\varphi(t)^{-1}|_{g_0}|g(t)^{-1}|_{g_0}|i\partial\bar{\partial}\varphi(t)|_{g_0}(x)\\
        \le &C_k\frac{t^k}{r(x)^{2k}},
    \end{split}
\end{equation*}
and,
\begin{equation*}
    \begin{aligned}
        |\tr_{\omega_\varphi(t)}\left(\mathcal{L}_X(i\partial\bar{\partial}\varphi(t))\right)|_{g_0}(x)&\le C(n)|g_{\varphi}(t)^{-1}|_{g_0}\left(|\nabla^{g_0}X|_{g_0}|i\partial\bar{\partial}\varphi(t)|_{g_0}+|\nabla^{g_0}i\partial\bar{\partial}\varphi(t)|_{g_0}|X|_{g_0})(x)\right)\\
        &\le C_k\frac{t^k}{r(x)^{2k}}.
    \end{aligned}
\end{equation*}
    Since $\lim_{t\to0^+}X\cdot\varphi(x,t)=0$, by integration we have $X\cdot\varphi(x,t)\le C_k\frac{t^{k+1}}{r(x)^{2k}}$ as required.

\end{proof}
\subsection{Obstruction tensor}
Recall that the self-similar solution to K\"ahler-Ricci flow $g(t)$ satisfies
\begin{equation}\label{soliton for g(t)}
    \mathcal{L}_{\frac{X}{2}}g(t)-t\Ric(g(t))-g(t)=0,\quad \textrm{for all $t\in \R_+$}.
\end{equation}
Hence, the obstruction tensor $T(t)$ to evaluate if $g_\varphi(t)$ coincides with $g(t)$ must be
\begin{equation}\label{obstruction tensor}
    T(t)=\mathcal{L}_{\frac{X}{2}}\omega_\varphi(t)-t\Ric(\omega_\varphi(t))-\omega_\varphi(t),\quad \textrm{for all $t\in (0,T)$}.
\end{equation}
Combining it with \eqref{soliton for g(t)}, we get
\begin{equation*}
    \begin{aligned}
        T(t)&=\mathcal{L}_{\frac{X}{2}}\omega_\varphi(t)-\mathcal{L}_{\frac{X}{2}}\omega(t)+t\Ric(\omega(t))-t\Ric(\omega_\varphi(t))+\omega(t)-\omega_\varphi(t)\\
        &=\mathcal{L}_{\frac{X}{2}}(i\partial\bar\partial\varphi(t))+ti\partial\bar\partial\log\frac{\omega_\varphi(t)^n}{\omega(t)^n}-i\partial\bar\partial\varphi(t)\\
        &=i\partial\bar\partial\left(\frac{X}{2}\cdot\varphi(t)+t\dot\varphi(t)-\varphi(t)\right).
    \end{aligned}
\end{equation*}

The obstruction tensor is very important in Riemannian geometry and K\"ahler geometry. In Riemannian geometry, analysis of this obstruction tensor is more complicated. In K\"ahler geometry, from the above calculation, $T$ is a (1,1) form given by a scalar function. Analysis of this scalar function may help us better understand the obstruction tensor.

\textbf{Define $u(x,t):=t\dot{\varphi}(x,t)+\frac{1}{2}X\cdot\varphi(x,t)-\varphi(x,t)$ for $(x,t)\in M\times (0,T)$.}
In Section \ref{soliton obstruction fl}, the function $u$ is related to $\partial_\tau\psi$ in self-similar variables (pulled back by the flow of $X$).
\begin{prop}[Solution to heat equation]\label{solution to heat equation}
   The function $u$ is a solution to the heat equation along Ricci flow, i.e.
    \begin{equation}
        \left(\frac{\partial}{\partial t}-\Delta_{\omega_\varphi(t)}\right)u(x,t)=0.
    \end{equation}
\end{prop}
\begin{proof} On the one hand
    \begin{equation*}
        \begin{aligned}
            \frac{\partial}{\partial t} u&=\dot{\varphi}+t\frac{\partial}{\partial t}\dot{\varphi}+\frac{1}{2}X\cdot\dot{\varphi}-\dot{\varphi
    }\\
    &=t\frac{\partial}{\partial t}\dot{\varphi}+\frac{1}{2}X\cdot\dot\varphi.
        \end{aligned}
    \end{equation*}
    On the other hand \begin{equation*}
        \frac{\partial}{\partial t}\dot\varphi=\Delta_{\omega_\varphi(t)}\dot\varphi+\tr_{\omega_\varphi(t)}(\frac{\partial}{\partial t}\omega(t))-\tr_{\omega(t)}(\frac{\partial}{\partial t}\omega(t)).
    \end{equation*}
    And 
\begin{equation*}
    \begin{aligned}
        X\cdot\dot{\varphi}&=X\cdot\log\frac{\omega_\varphi(t)^n}{\omega(t)^n}\\
        &=\tr_{\omega_\varphi(t)}(\mathcal{L}_X\omega_\varphi(t))-\tr_{\omega(t)}(\mathcal{L}_X \omega(t))\\
        &=\tr_{\omega_\varphi(t)}(\mathcal{L}_X \omega(t))-\tr_{\omega(t)}(\mathcal{L}_X \omega(t))+\Delta_{\omega_\varphi(t)}(X\cdot\varphi).
    \end{aligned}
\end{equation*}
By \eqref{soliton for g(t)} and $\frac{\partial}{\partial t}\omega(t)=-\rho_{\omega(t)}$, we have
\begin{equation*}
    \begin{aligned}
        \frac{\partial}{\partial t} u&=\Delta_{\omega_\varphi(t)}(u+\varphi)+\tr_{\omega_\varphi(t)}(\omega(t))-\tr_{\omega(t)}(\omega(t))\\
        &=\Delta_{\omega_\varphi(t)}(u+\varphi)-\Delta_{\omega_\varphi(t)}\varphi=\Delta_{\omega_\varphi(t)}u
    \end{aligned}
\end{equation*}
as desired.
\end{proof}
The obstruction tensor is exactly $i\partial\bar\partial u$. Moreover, $u$ is a solution to the heat equation along the K\"ahler-Ricci flow.

We immediately obtain rough estimates for $u,X\cdot u$ and $\partial\bar{\partial}u$ thanks to \eqref{polynomial decay formula}.
\begin{corollary}\label{coro: estimates on u}
For $\lambda>0$ as in Proposition \ref{polynomial estimate of Hessian and gradient of Hessian}, there exists a constant $N>0$ independent of time, such that for all $(x,t)\in\Omega_\lambda$, \begin{equation}\label{rough estimates of u}
        \begin{split}
            |u(x,t)|&\le N\frac{t^2}{r(x)^2},\\
            |X\cdot u(x,t)|&\le N\frac{t^2}{r(x)^2},\\
            |i\partial\bar{\partial}u(x,t)|_{g_0}&\le N\frac{t}{r(x)^2}.\\
        \end{split}
    \end{equation}
\end{corollary}
\begin{proof}
For convenience, $N>0$ is a constant independent of time that may vary from line to line.

    The estimate of $u$ follows immediately by \eqref{polynomial decay formula}. Due to \eqref{polynomial decay formula}, there exists a constant $N$ such that $|X\cdot\dot\varphi|\le N\frac{t}{r(x)^2}$.
    
    Moreover, since $JX\cdot\varphi=0, JX\cdot(X\cdot\varphi)=0$ and $X$ is real holomorphic, we have
    \begin{equation*}
       - i\partial\bar{\partial}{\varphi}(JX,X)=2X^iX^{\bar j}\partial_i\bar\partial_{\bar j}\varphi=2X^i\partial_i(X^{\bar j}\bar\partial_{\bar j}\varphi)=X^i\partial_i(X\cdot\varphi)=\frac{1}{2}X\cdot(X\cdot\varphi).
    \end{equation*}
    Combining it with \eqref{polynomial estimate of Hessian and gradient of Hessian} and \eqref{polynomial decay formula}, there is a constant $N>0$ such that $|X\cdot u(x,t)|\le N\frac{t^2}{r(x)^2}$.

    Since $g_\varphi(t)$ and $g(t)$ are two solutions emerging from cone, there exists $N>0$ such that \begin{equation*}
        |\Ric(g_\varphi(t))(x)|_{g_0}+|\Ric(g(t))(x)|_{g_0}\le \frac{N}{r(x)^2}
    \end{equation*} for $(x,t)\in\Omega_\lambda$. Recall that
    \begin{equation*}
        \begin{aligned}
            i\partial\bar{\partial}u(x,t)&=ti\partial\bar{\partial}\log\frac{\omega_\varphi(t)^n}{\omega(t)^n}+i\partial\bar{\partial}\left(\frac{1}{2}X\cdot\varphi(t)\right)-i\partial\bar{\partial}\varphi(t)\\
            &=t\Ric(\omega(t))(x)-t\Ric(\omega_\varphi(t))(x)+\mathcal{L}_{\frac{X}{2}}(i\partial\bar{\partial}\varphi(t))-i\partial\bar{\partial}\varphi(t).\\
        \end{aligned}
    \end{equation*}
    By \eqref{polynomial estimate of Hessian and gradient of Hessian} and the proof of \eqref{polynomial decay formula}, we have $|i\partial\bar{\partial}u(x,t)|_{g_0}\le N\frac{t}{r(x)^2}$ as required.
\end{proof}
\subsection{Hamiltonian potentials of the soliton vector}
Let $f\in C^{\infty}(M;\R)$ be the normalized Hamiltonian potential of $X$ with respect to the K\"ahler-Ricci soliton metric $g$. Then for any $t\in(0,T)$, function $t\Phi_t^*f$ is a Hamiltonian potential of $X$ with respect to $g(t)$.
\begin{lemma}
   For $t\in (0,T)$, $t\Phi_t^*f+\frac{1}{2}X\cdot\varphi(t)$ is a Hamiltonian potential of $X$ with respect to $g_{\varphi}(t)$.
\end{lemma}
\begin{proof}
    Choose a local holomorphic coordinate of $M$, since $X$ is real holomorphic and $JX\cdot\varphi=0$, $X^{1,0}\cdot\varphi=\frac{1}{2}X\cdot\varphi$, we have
    \begin{equation*}
        \begin{aligned}
            \bar{\partial}_{\bar{j}}(t\Phi_t^*f+\frac{1}{2}X\cdot\varphi(t))&=\bar{\partial}_{\bar{j}}(t\Phi_t^*f)+\bar{\partial}_{\bar{j}}(X^{1,0}\cdot\varphi(t))\\
            &=g(t)_{i\bar{j}}X^{i}+(X^{1,0})^i\partial_i\bar{\partial}_{\bar{j}}\varphi(t)\\
            &=g(t)_{i\bar{j}}X^{i}+X^{i}\partial_i\bar{\partial}_{\bar{j}}\varphi(t)\\
            &=g_{\varphi}(t)_{i\bar{j}}X^i.
        \end{aligned}
    \end{equation*}
    Hence, $t\Phi_t^*f+\frac{1}{2}X\cdot\varphi(t)$ is a Hamiltonian potential of $X$ with respect to $g_{\varphi}(t)$.
\end{proof}
\textbf{For convenience, we define $f_t(x,t):=\Phi_t^*f(x),f_\varphi(x,t):=\Phi_t^*f(x)+\frac{1}{2t}X\cdot\varphi(x,t)$ for $(x,t)\in M\times(0,T)$.}
\begin{corollary}
For $\lambda>0$ as in Proposition \ref{consequence of taylor expansion}, there exists a constant $D>0$ independent of time such that:
    \begin{equation}\label{equivalence of potential formula}
        \begin{split}
            &\left|f_t(x,t)-\frac{r(x)^2}{2t}\right|\le D,\quad \forall (x,t)\in M\times (0,T)\\
            &\left|f_\varphi(x,t)-\frac{r(x)^2}{2t}\right|\le D,\quad \forall (x,t)\in\Omega_\lambda.\\
        \end{split}
    \end{equation}
\end{corollary}
\begin{proof}
In this proof $D$ is a constant independent of time that varies from line to line.

    From Lemma \ref{soliton equalities}, we deduce that
    \begin{equation*}
        \Phi_t^*f=\Phi^*_t(|\partial f|^2_g+R_\omega)=|\partial \Phi_t^*f|^2_{\Phi_t^*g}+\Phi^*_tR_\omega=t^{-1}|\partial t\Phi_t^*f|^2_{g(t)}+\Phi^*_tR_\omega=t^{-1}\frac{1}{2}|X|^2_{g(t)}+\Phi^*_tR_\omega.
    \end{equation*}
    Since $R_\omega$ is globally bounded, we bound it by $D$, and we notice that in $\Omega_\lambda$, 
    \begin{equation*}
        \begin{aligned}
            \left||X|^2_{g(t)}-|X|^2_{g_0}\right|&\le2\int_0^t |\Ric(g(s))|_{g_0}|X|^2_{g_0}ds
           \le Dt.
        \end{aligned}
    \end{equation*}
    Hence, we get $|\Phi_t^*f(x)-\frac{r(x)^2}{2t}|\le D$ on $\Omega_\lambda$. Take $t=1$, and increase $D$ if necessary, since $r^2$ is a continuous function on $M$, we have $|f-\frac{r^2}{2}|\le D$ on $M$. After pulling back by $\Phi_t$, we have
    \begin{equation*}
        \left|f_t-\frac{r^2}{2t}\right|\le D,
    \end{equation*}
    holds on $M\times (0,T)$.
    
    The second inequality comes from the polynomial decay of $t^{-1}X\cdot\varphi$ stated in \eqref{polynomial decay formula}.
\end{proof}
\begin{prop}\label{evolution equation of potentials}
  Hamiltonian potentials $f_t$ and $f_\varphi$ enjoy the following parabolic evolution equations:
    \begin{equation}
        \begin{split}
        &\left(\frac{\partial}{\partial t}-\Delta_{\omega(t)}\right)f_t(x,t)=-\frac{n+f_t(x,t)}{t},\quad \textrm{for all $(x,t)\in M\times (0,T)$},\\
            &\left(\frac{\partial}{\partial t}-\Delta_{\omega_\varphi(t)}\right)f_\varphi(x,t)=-\frac{n+f_\varphi(x,t)}{t},\quad \textrm{for all $(x,t)\in M\times (0,T)$}.\\
        \end{split}
    \end{equation}
\end{prop}
\begin{proof}
    By soliton identities in Lemma \ref{soliton equalities}:
   \begin{equation*}
       \begin{aligned}
           \frac{\partial}{\partial t} \left(tf_t(x,t)\right)&=f_t(x,t)-\frac{1}{2}X(f_t)(x,t)\\
        &=f(\Phi_t(x))-\frac{1}{2}|X|^2_{g}(\Phi_t(x))\\
        &=R_{\omega}(\Phi_t(x))\\
        &=(\Delta_{\omega}f)(\Phi_t(x))-n\\
        &=\Delta_{\Phi_t^*\omega}f_t(x,t)-n\\
        &=t\Delta_{\omega(t)}f_t(x,t)-n.\\
       \end{aligned}
   \end{equation*}
        This proves the first expected result. Moreover,
         \begin{equation*}
             \begin{aligned}
                 \frac{\partial}{\partial t} f_\varphi&=\frac{\partial}{\partial t} f_t-\frac{1}{2t^2}X\cdot\varphi+\frac{1}{2t}X\cdot\dot{\varphi}\\
       &=\frac{\partial}{\partial t} f_t-\frac{1}{2t^2}X\cdot\varphi+\frac{1}{2t}X\cdot \log\frac{\omega_\varphi(t)^n}{\omega(t)^n}\\
       &=\frac{\partial}{\partial t} f_t-\frac{1}{2t^2}X\cdot\varphi+\frac{1}{2t}\left(\tr_{\omega_\varphi(t)}(\mathcal{L}_X\omega_\varphi(t))-\tr_{\omega(t)}(\mathcal{L}_X\omega(t))\right)\\
       &=\frac{\partial}{\partial t} f_t-\frac{1}{2t^2}X\cdot\varphi+\frac{1}{t}(\Delta_{\omega_\varphi(t)}tf_\varphi-\Delta_{\omega(t)}tf_t)\\
       &=\frac{\partial}{\partial t}f_t-\Delta_{\omega(t)}f_t-\frac{1}{2t^2}X\cdot\varphi+\Delta_{\omega_\varphi(t)}f_\varphi.
             \end{aligned}
         \end{equation*}
         By the first identity, $\left(\frac{\partial}{\partial t}-\Delta_{\omega(t)}\right)f_t(x,t)=-\frac{n+f_t(x,t)}{t}$, we have
         \begin{equation*}
            \frac{\partial}{\partial t} f_\varphi= -\frac{n+f_t(x,t)}{t}-\frac{1}{2t^2}X\cdot\varphi+\Delta_{\omega_\varphi(t)}f_\varphi=-\frac{n+f_\varphi(x,t)}{t}+\Delta_{\omega_\varphi(t)}f_\varphi.
         \end{equation*}
         as expected.
    \end{proof}
\subsection{Maximum principle}\label{section max principle}
Since $u$ decays at infinity faster than any polynomial in $\frac{t}{r^2}$, as established by \eqref{polynomial decay formula}, we are led to believe that the decay of $ u $ should in fact be exponential. In this section, we rigorously confirm this intuition by applying the maximum principle in combination with a carefully constructed auxiliary function. Specifically, we demonstrate that
\begin{equation*}
    u(x,t) =O\left(t e^{-\frac{r^2}{2t}} \left(\frac{r^2}{2t}\right)^{-n-1} \right),
\end{equation*}
as $\frac{r^2}{t} \to \infty$. Moreover, the analysis shows that the exponent $n+1$ in the polynomial factor is essentially optimal.
\begin{prop}[Barrier functions]\label{barrier function}
For $\lambda>0$ as in Proposition \ref{consequence of taylor expansion}, there exist constants $\lambda'>\lambda,\beta>0$ independent of time such that 
\begin{equation*}
    \begin{split}
        &v(x,t):=tf_\varphi^{-n-1}e^{-f_\varphi-\beta f_\varphi^{-1}}(x,t)\quad \textrm{and}\\
        &\bar{v}(x,t):=tf_t^{-n-1}e^{-f_t-\beta f_t^{-1}}(x,t)
    \end{split}
\end{equation*} are well defined on $\Omega_{\lambda'}$ and satisfy
    \begin{equation}
        \begin{split}
            \left(\frac{\partial}{\partial t}-\Delta_{\omega_\varphi(t)}\right)v(x,t)&\ge 0, \quad \forall (x,t)\in\Omega_{\lambda'},\ t\neq0,\\
            \left(\frac{\partial}{\partial t}-\Delta_{\omega(t)}\right)\bar{v}(x,t)&\ge 0, \quad \forall (x,t)\in\Omega_{\lambda'},\ t\neq0.
        \end{split}
    \end{equation}
\end{prop}
\begin{proof}
    For convenience, we omit the time dependence of $\omega_\varphi(t)$, we simply denote it by $\omega_\varphi$. 
    
First, set $\lambda_1=\max\{\lambda,4D+2\}$ with $D$ in \eqref{equivalence of potential formula}, $\Omega_{\lambda_1}$ is a subset of $\Omega_\lambda$, estimates \eqref{polynomial estimate of Hessian and gradient of Hessian} \eqref{polynomial decay formula} \eqref{rough estimates of u} and \eqref{equivalence of potential formula} still hold in $\Omega_{\lambda_1}$. The choice of $\lambda_1$ ensures that in $\Omega_{\lambda_1}$
\begin{equation*}
 \begin{split}
       &\frac{r(x)^2}{4t}\le f_t(x,t)\le \frac{r(x)^2}{t},\\
         &\frac{r(x)^2}{4t}\le f_\varphi(x,t)\le \frac{r(x)^2}{t}.\\
 \end{split}
\end{equation*}

    Consequently, $v$ is well defined on $\Omega_{\lambda_1}$ and if we fix $x\in M$ such that $r(x)\neq0$, $\frac{v(x,t)}{t}$ converges to 0 locally when $t$ tends to 0.

Recall the evolution equation of $f_\varphi$ (Proposition \ref{evolution equation of potentials}) is
    \begin{equation*}
        \left(\frac{\partial}{\partial t}-\Delta_{\omega_\varphi}\right)f_\varphi=-\frac{n+f_\varphi}{t}.
    \end{equation*}
    Now we compute the evolution equation of $v$ at point $(x,t)\in\Omega_{\lambda_1}.$ Set $v_1:=tf_\varphi^{-n-1}e^{-f_\varphi},v_2:=e^{-\beta f_\varphi^{-1}}$ and compute as follows:
    \begin{equation*}
        \begin{aligned}
         \frac{\partial}{\partial t}v_1&=-v_1(1+(n+1) f_\varphi^{-1})\frac{\partial}{\partial t} f_\varphi +\frac{v_1}{t},\\
            \Delta_{\omega_\varphi}v_1&=t\Delta_{\omega_\varphi}(e^{-f_\varphi})f_\varphi^{-n-1}+t\Delta_{\omega_\varphi}(f_\varphi^{-n-1})e^{-f_\varphi}+2t\Re(\langle\partial e^{-f_\varphi},\bar{\partial}f_\varphi^{-n-1}\rangle_{\omega_\varphi})\\
        &=v_1\left[|\partial f_\varphi|_{\omega_\varphi}^2-(1+(n+1)f_\varphi^{-1})\Delta_{\omega_\varphi}f_\varphi+(n+2)(n+1)|\partial f_\varphi|_{\omega_\varphi}^2f_\varphi^{-2}+2(n+1)|\partial f_\varphi|_{\omega_\varphi}^2f_{\varphi}^{-1}\right].
        \end{aligned}
    \end{equation*}
    Hence
    \begin{equation}\label{v_1}
        \left(\frac{\partial}{\partial t}-\Delta_{\omega_\varphi}\right)v_1=v_1\left[\frac{1}{t}+(1+(n+1)f_\varphi^{-1}))\frac{n+f_\varphi}{t}-|\partial f_\varphi|_{\omega_\varphi}^2(1+(n+2)(n+1)f_\varphi^{-2}+2(n+1)f_\varphi^{-1})\right].
    \end{equation}
    Recall that $|\partial f_\varphi|_{\omega_\varphi}^2(x,t)=\frac{|X|^2_{\omega_\varphi}(x,t)}{2t^2}$, and locally
    \begin{equation*}
        \begin{aligned}
            |X|^2_{\omega_\varphi}&=|X|^2_{\omega(t)}+2\partial_{i}\partial_{\bar{j}}\varphi X^iX^{\bar{j}}\\
        &=2tf_t-2tR_\omega\circ\Phi_t+2\partial_{i}\partial_{\bar{j}}\varphi X^iX^{\bar{j}}\\
        &=2tf_\varphi-X\cdot\varphi-2tR_\omega\circ\Phi_t+2\partial_{i}\partial_{\bar{j}}\varphi X^iX^{\bar{j}}\\
        &:=2tf_\varphi+tS.\\
        \end{aligned}
    \end{equation*}
    Here $S:=-\frac{1}{t}X\cdot\varphi-2R_\omega\circ\Phi_t+\frac{2}{t}\partial_{i}\partial_{\bar{j}}\varphi X^iX^{\bar{j}}$ and the second line is ensured by soliton identities (Lemma \ref{soliton equalities}). Hence
    \begin{equation*}
        |\partial f_\varphi|_{\omega_\varphi}^2=\frac{f_\varphi}{t}+\frac{S}{2t}.
    \end{equation*}
    Therefore we rewrite \eqref{v_1} as follows:
    \begin{equation}\label{v_1 with S}
        \begin{aligned}
            \left(\frac{\partial}{\partial t}-\Delta_{\omega_\varphi}\right)v_1&=v_1\left[\frac{1}{t}+(1+(n+1)f_\varphi^{-1}))\frac{n+f_\varphi}{t}-\frac{f_\varphi}{t}(1+(n+2)(n+1)f_\varphi^{-2}+2(n+1)f_\varphi^{-1})\right]\\
        &\quad-v_1\left[\frac{S}{2t}(1+(n+2)(n+1)f_\varphi^{-2}+2(n+1)f_\varphi^{-1})\right]\\
        &=v_1\left[-\frac{2n+2}{tf_\varphi}-\frac{S}{2t}(1+(n+2)(n+1)f_\varphi^{-2}+2(n+1)f_\varphi^{-1})\right].\\
        \end{aligned}
    \end{equation}
    Now we compute the evolution equation of $v_2$ at the point $(x,t)\in\Omega_{\lambda_1}$:
    \begin{equation}\label{v_2}
    \begin{aligned}
        \frac{\partial}{\partial t} v_2&=\beta f_\varphi^{-2}v_2\frac{\partial}{\partial t} f_\varphi,\\
         \Delta_{\omega_\varphi}v_2&=v_2(\beta f_\varphi^{-2}\Delta_{\omega_\varphi}f_\varphi+\beta^2f_{\varphi}^{-4}|\partial f_\varphi|_{\omega_\varphi}^2-2\beta|\partial f_\varphi|_{\omega_\varphi}^2f_\varphi^{-3}),\\
         \left(\frac{\partial}{\partial t}-\Delta_{\omega_\varphi}\right)v_2&=v_2\left[-\beta \frac{n+f_\varphi}{tf_\varphi^{2}}-\beta^2f_\varphi^{-4}|\partial f_\varphi|_{\omega_\varphi}^2+2\beta|\partial f_\varphi|_{\omega_\varphi}^2f_\varphi^{-3}\right]\\
         &=v_2\left[-\frac{\beta}{tf_\varphi}+\frac{(2-n)\beta}{tf_\varphi^2}-\frac{\beta^2}{tf_\varphi^3}+\frac{S}{2t}(2\beta f_\varphi^{-3}-\beta^2f_\varphi^{-4})\right].
    \end{aligned}
    \end{equation}
    The scalar product of $\partial v_1$ and $\bar{\partial}v_2$ is given by
   \begin{equation}\label{sclar product}
       \begin{aligned}
           \langle\partial v_1,\bar{\partial} v_2\rangle_{\omega_\varphi}&=-\beta(1+(n+1)f_\varphi^{-1})f_\varphi^{-2}|\partial f_\varphi|_{\omega_\varphi}^2v_1v_2\\
           &=v_1v_2\left[-\frac{\beta}{tf_\varphi}-\frac{(n+1)\beta}{tf_\varphi^2}+\frac{S}{2t}(-\beta f_\varphi^{-2}-(n+1)\beta f_\varphi^{-3})\right].
       \end{aligned}
   \end{equation}
   By definition $v=v_1v_2$, combining \eqref{v_1 with S} \eqref{v_2} and \eqref{sclar product}, we have
   \begin{equation}\label{v}
       \begin{aligned}
           \left(\frac{\partial}{\partial t}-\Delta_{\omega_\varphi}\right)v&=v_1\left(\frac{\partial}{\partial t}-\Delta_{\omega_\varphi}\right)v_2+v_2\left(\frac{\partial}{\partial t}-\Delta_{\omega_\varphi}\right)v_1-2\Re(\langle\partial v_1,\bar{\partial} v_2\rangle_{\omega_\varphi})\\
        &=v\left[\frac{\beta-2n-2}{tf_\varphi}+\frac{(n+4)\beta}{tf_\varphi^2}-\frac{\beta^2}{tf_\varphi^3}\right]\\
        &\quad+v\frac{S}{2t}\left[\frac{2(n+2)\beta}{f_\varphi^3}+\frac{2\beta}{f_\varphi^2}-\frac{\beta^2}{f_\varphi^4}-1-\frac{(n+2)(n+1)}{f_\varphi^2}-\frac{2n+2}{f_\varphi}\right].\\
       \end{aligned}
   \end{equation}
    Now we estimate $\frac{S}{2t}$, by \eqref{polynomial estimate of Hessian and gradient of Hessian} and \eqref{polynomial decay formula},
    \begin{align*}
        |S(x,t)|&=\left|-2R_\omega(\Phi_t(x))-\frac{X(\varphi)(x,t)}{t}+2\frac{\partial_i\bar{\partial}_{\bar{j}}\varphi X^iX^{\bar{j}}(x,t)}{t}\right|\\
        &\le 2|R_\omega(\Phi_t(x))|+C_1\frac{t}{r(x)^2}+B_2\frac{t^2}{r(x)^4}\frac{r(x)^2}{t}\\
        &\le 2|R_\omega(\Phi_t(x))|+(C_1+B_2)\frac{t}{r(x)^2}.\\
    \end{align*}
    Since scalar curvature decays quadratically, by \cite[Proposition 2.19]{Conlon-Deruelle-Sun} there exists a constant $C$ such that
    \begin{equation*}
        |R_\omega(y)|\le \frac{C}{4|f(y)|},\quad \forall y\in M.
    \end{equation*}
    Hence $|R_\omega(\Phi_t(x))|\le\frac{C}{4|f_t(x)|}$. Moreover for $(x,t)\in\Omega_{\lambda_1}$, $f_t(x)>\frac{r(x)^2}{4t}$, so that $|R_\omega(\Phi_t(x))|\le C\frac{t}{r(x)^2}$.
   
   Let $2C':=2C+B_2+C_1$, we have on $\Omega_{\lambda_1}$
    \begin{equation*}
        \left|\frac{S}{2t}(x,t)\right|\le C'\frac{1}{r(x)^2}.
    \end{equation*}
    Notice in $\Omega_{\lambda_1}$, $f_\varphi(x,t)\le \frac{r(x)^2}{t},f_\varphi(x,t)\ge \frac{1}{4}\frac{r(x)^2}{t}$.
    Take $\beta=2n+2+(C'+1)$, from \eqref{v} we have
    \begin{equation*}
        \begin{aligned}
            \left(\frac{\partial}{\partial t}-\Delta_{\omega_\varphi}\right)v&\ge v\left[\frac{\beta-2n-2}{r(x)^2}-\frac{\beta^2}{tf_\varphi^3}-\frac{C'}{r(x)^2}\left(\frac{2(n+2)\beta}{f_\varphi^3}+\frac{2\beta}{f_\varphi^2}+\frac{\beta^2}{f_\varphi^4}+1+\frac{(n+2)(n+1)}{f_\varphi^2}+\frac{2n+2}{f_\varphi}\right)\right]\\ 
        &\ge v\left[\frac{{\beta-2n-2}-C'}{r(x)^2}-\frac{\beta^2}{tf_\varphi^3}-\frac{C'}{r(x)^2}\left(\frac{2(n+2)\beta}{f_\varphi^3}+\frac{2\beta}{f_\varphi^2}+\frac{\beta^2}{f_\varphi^4}+\frac{(n+2)(n+1)}{f_\varphi^2}+\frac{2n+2}{f_\varphi}\right) \right]\\
        &\ge v\left[\frac{1}{r(x)^2}-\frac{\beta^2}{tf_\varphi^3}-\frac{C'}{r(x)^2}\left(\frac{2(n+2)\beta}{f_\varphi^3}+\frac{2\beta}{f_\varphi^2}+\frac{\beta^2}{f_\varphi^4}+\frac{(n+2)(n+1)}{f_\varphi^2}+\frac{2n+2}{f_\varphi}\right) \right]\\
        &\ge v\left[\frac{1}{r(x)^2}-\frac{4}{r(x)^2}\frac{\beta^2}{f_\varphi^2}-\frac{C'}{r(x)^2}\left(\frac{2(n+2)\beta}{f_\varphi^3}+\frac{2\beta}{f_\varphi^2}+\frac{\beta^2}{f_\varphi^4}+\frac{(n+2)(n+1)}{f_\varphi^2}+\frac{2n+2}{f_\varphi}\right) \right]\\
        &=\frac{v}{r(x)^2}\left[1-\frac{1}{f_\varphi}\left(\frac{4\beta^2+2C'\beta+C'(n+2)(n+1)}{f_\varphi}+\frac{2C'(n+2)\beta}{f_\varphi^2}+\frac{C'\beta^2}{f_\varphi
        ^3}+C'(2n+2)\right)\right].\\
        \end{aligned}
    \end{equation*}
    In $\Omega_{\lambda_1}$ $f_\varphi(x,t)\ge \frac{1}{4}\frac{r(x)^2}{t}.$ Therefore there exists a constant $\lambda'\ge\lambda_1$ large enough such that for any $(x',t')\in\Omega_{\lambda'}$, we have $f_\varphi(x',t')\ge \frac{1}{4}\frac{r(x')^2}{t'}\ge\frac{\lambda'}{4}$ and
    \begin{equation*}
        \begin{aligned}
            &\frac{1}{f_\varphi}\left[\frac{4\beta^2+2C'\beta+C'(n+2)(n+1)}{f_\varphi}+\frac{2C'(n+2)\beta}{f_\varphi^2}+\frac{C'\beta^2}{f_\varphi
        ^3}+C'(2n+2)\right](x',t')\\
        \le& \frac{4}{\lambda'}\left[4\left(\frac{4\beta^2+2C'\beta+C'(n+2)(n+1)}{\lambda'}\right)+16\left(\frac{2C'(n+2)\beta}{\lambda'^2}\right)+64\left(\frac{C'\beta^2}{\lambda'^3}\right)+C'(2n+2)\right]\le 1.\\
        \end{aligned}
    \end{equation*}
    Therefore for $(x,t)\in \Omega_{\lambda'}$ we have
    \begin{equation*}
        \left(\frac{\partial}{\partial t}-\Delta_{\omega_\varphi(t)}\right)v(x,t)\ge 0.
    \end{equation*}

    The calculation of $\bar{v}$ follows verbatim since $\left(\frac{\partial}{\partial t}-\Delta_{\omega(t)}\right)f_t(x,t)=-\frac{n+f_t}{t}$.
\end{proof}
\textbf{We henceforth use $\lambda$ to denote the larger constant $\lambda'$ in Proposition \ref{barrier function}.  All the previous estimates in \eqref{polynomial estimate of Hessian and gradient of Hessian} \eqref{polynomial decay formula} \eqref{rough estimates of u} and \eqref{equivalence of potential formula} still hold on $\Omega_{\lambda_1}$ because $\Omega_{\lambda'}$ is a subset of $\Omega_\lambda$.}
\begin{lemma}[Maximum principle]\label{maximum principle}
    Let $g(t)_{t\in (0,T)}$ be a smooth family of Riemannian metrics defined on $M\times (0,T)$. Let $\kappa$ be a continuous function defined on $\{(x,t)\ |\ r(x)^2\ge\lambda t,\ t\neq0\}$ and smooth in the interior of $\Omega_\lambda$. Assume that:
    \begin{enumerate}
        \item $\left(\frac{\partial}{\partial t}-\Delta_{g(t)}\right)\kappa(x,t)\le0$ for all $(x,t)$ in the interior of $\Omega_\lambda$.
        \item There exists a constant $C>0$ such that $\kappa(x,t)\le Ct$ for all $(x,t)\in\Omega_\lambda$.
        \item For all $0<T_0<T_1<T$, $\lim_{r(x)\to+\infty}\sup_{t\in [T_0,T_1]} \kappa(x,t)\le 0$.
        \item $\kappa(x,t)\le0$ for all $(x,t)\in M\times(0,T)$ such that $r(x)^2=\lambda t$.
    \end{enumerate}
    Then $\kappa\le 0$ in $\Omega_\lambda$.
\end{lemma}
\begin{proof}
For $(x,t)\in \Omega_\lambda$ and $t\neq0$,  take $0<t'<t$ and define $\Gamma_t:=\{(x,s)\ |\ r(x)^2\ge\lambda s,\ s\in [t',t]\}$.

For any $\varepsilon>0$, consider $\kappa_\varepsilon(x,t):=\kappa(x,t)-\varepsilon t$, and by assumption (i) $\left(\frac{\partial}{\partial t}-\Delta_{g(t)})\kappa_\varepsilon(x,t\right)\le-\varepsilon<0$.

If $\sup_{\Gamma_t}\kappa_\varepsilon\le0$, then in particular $\kappa_\varepsilon(x,t)\le 0$.

If $\sup_{\Gamma_t}\kappa_\varepsilon>0$, by assumption (iii), there exists $(x_0,t_0)$ in the closure of $\Gamma_t$ such that $\kappa_\varepsilon(x_0,t_0)=\sup_{\Gamma_t}\kappa_\varepsilon$. Since $\left(\frac{\partial}{\partial t}-\Delta_{g(t)})\kappa_\varepsilon(x,t\right)\le-\varepsilon<0$, by the weak maximum principle, $(x_0,t_0)$ must lie on the parabolic boundary, i.e $(x_0,t_0)\in \{(y,s)\ |\ r(y)^2=\lambda s,\ s\in [t',t]\}\cup\{(y,t')\ |\ r(y)^2>\lambda t'\}$. By assumption (iv), we conclude that $(x_0,t_0)\in \{(y,t')\ |\ r(y)^2>\lambda t'\}$. Thanks to assumption (ii), we have in particular,
\begin{equation*}
    \kappa_\varepsilon(x,t)\le\kappa_{\varepsilon}(x_0,t_0)\le Ct'.
\end{equation*}
Letting $t'\to 0$ and then $\varepsilon\to 0$ yields to the desired result.
\end{proof}
\begin{corollary}[Exponential decay]\label{exponential decay}
  For $\lambda>0$ as in Proposition \ref{barrier function}, there exists a constant $C>0$ independent of time such that for all $(x,t)\in\Omega_\lambda$,
    \begin{equation}
        \begin{split}
        &|\varphi(x,t)|\le Cte^{-f_t}f_t^{-n-1}(x,t),\\
            &|u(x,t)|\le Cte^{-f_t}f_t^{-n-1}(x,t).
        \end{split}
    \end{equation}
\end{corollary}
\begin{proof}
   
Due to Corollary \ref{coro: estimates on u}, $|u(x,t)|\le N\frac{t^2}{r(x)^2}\le N\frac{1}{\lambda}t$ holds on $\Omega_\lambda$. If $r(x)^2=\lambda t$, then by \eqref{equivalence of potential formula}, $0<\frac{\lambda}{2}-D\le f_\varphi(x,t)\le \frac{\lambda}{2}+D$, by the choice of $\lambda$ , there exists an $\varepsilon>0$ such that $v(x,t)\ge\varepsilon t$ for $v$ defined as in Proposition \ref{barrier function} if $r(x)^2=\lambda t$. Proposition \ref{solution to heat equation} and Proposition \ref{barrier function} ensure that:
\begin{enumerate}
    \item $\left(\frac{\partial}{\partial t}-\Delta_{\omega_\varphi(t)}\right)(u-\frac{N\frac{1}{\lambda}}{\varepsilon}v)\le 0$ in the interior of $\Omega_\lambda$.
    \item $(u-\frac{N\frac{1}{\lambda}}{\varepsilon}v)(x,t)\le \frac{N}{\lambda}t$ on $\Omega_\lambda$.
    \item For all $0<T_0<T_1<T$, we have
    \begin{equation*}
        \lim_{r(x)\to+\infty}\sup_{t\in [T_0,T_1]}(u-\frac{N\frac{1}{\lambda}}{\varepsilon}v)(x,t)\le\lim_{r(x)\to+\infty}\sup_{t\in [T_0,T_1]}N\frac{t^2}{r(x)^2}=0.
    \end{equation*}
    \item $(u-\frac{N\frac{1}{\lambda}}{\varepsilon}v)(x,t)\le0$ for all $(x,t)\in M\times (0,T)$ such that $r(x)^2=\lambda t$.
\end{enumerate}
By Lemma \ref{maximum principle}, $u\le \frac{N\frac{1}{\lambda}}{\varepsilon}v$ in $\Omega_\lambda$. The proof of $u\ge -\frac{N\frac{1}{\lambda}}{\varepsilon}v$ follows analogously. Define $C=\frac{N\frac{1}{\lambda}}{\varepsilon}$, on $\Omega_\lambda$, we have $|u(x,t)|\le Cv(x,t)\le Cte^{-f_\varphi}f_\varphi^{-n-1}(x,t)$. By \eqref{equivalence of potential formula} and the choice of $\lambda$, we deduce that there exists a constant $C>0$ such that
$|u(x,t)|\le Cte^{-f_t}f_t^{-n-1}(x,t)$.

As for estimating $\varphi$, we observe that $\frac{\partial}{\partial t}\varphi\le \Delta_{\omega(t)}\varphi$ and $\frac{\partial}{\partial t}\varphi\ge\Delta_{\omega_\varphi(t)}\varphi$ thanks to \eqref{Monge-Ampere equation}. The same argument involving $\varphi$ and $v,\bar{v}$ implies that there exists $C>0$ such that $|\varphi(x,t)|\le Cte^{-f_t}f_t^{-n-1}(x,t)$.
\end{proof}
From \eqref{equivalence of potential formula}, we observe that the set $\{f_t \ge \frac{\lambda}{2} + D\}$ is a subset of $\Omega_\lambda$. \textbf{For convenience, we will henceforth use $\Omega_\lambda$ to denote the set $\{f_t \ge \frac{\lambda}{2} + D\}$, and, with slight abuse of notation, we continue to write $\lambda$ instead of $\frac{\lambda}{2} + D$.} With this convention, the above estimates hold within $\Omega_\lambda$, let us summarize them as follows:
\begin{prop}\label{resume of estimate}
There exists a constant $\lambda>0$, such that for $(x,t)\in \Omega_\lambda:=\{(x,t)\ |\ f_t(x,t)\ge\lambda\}\subset M\times (0,T)$, we have
\begin{enumerate}
    \item $|u(x,t)|\le Cte^{-f_t}f_t^{-n-1}(x,t)$,
    \item $|\varphi(x,t)|\le Cte^{-f_t}f_t^{-n-1}(x,t)$,
    \item $|X\cdot\varphi(x,t)|\le C_ktf_t^{-k}(x,t)$ for all $k\in\N_0$,
    \item $|X\cdot u(x,t)|\le Ntf_t^{-1}$,
    \item $|i\partial\bar{\partial}u(x,t)|_{g_0}\le Nf_t^{-1}$,
    \item $A_0^{-1}g(t)\le g_\varphi(t)\le A_0g(t)$, and $A_0^{-1}g_0\le g_\varphi(t)\le A_0g_0$ for some uniform constant $A_0>1$.
\end{enumerate}
\end{prop}
\section{Analysis of the normalized K\"ahler-Ricci flow}\label{soliton obstruction fl}
In this section, we consider the following smooth space-time function:
\begin{equation*}
    \varphi(x,t)=t\psi(\Phi_t(x),\tau),\quad\textrm{for $(x,t)\in M\times (0,T)$,\quad $\tau=\log t.$}
\end{equation*}
Recall the setting of $\varphi$ as in Proposition \ref{Proposition monge ampere}:
\begin{equation*}
    \begin{split}
        &\frac{\partial}{\partial t}\varphi(t)=\log\frac{\omega_\varphi(t)^n}{\omega(t)^n},\\
        & JX\cdot\varphi(t)=0,\quad \forall t\in (0,T),
    \end{split}
\end{equation*}
with $\pi_*\varphi(t)$ converging to $0$ locally smoothly when $t$ tends to $0$.

We deduce that $\psi$ satisfies the following evolution equation
\begin{equation*}
    \begin{aligned}
        t\Phi_t^*(\frac{\partial}{\partial \tau}\psi)+t\Phi_t^*\psi-\frac{t}{2}X\cdot\Phi_t^*\psi=t\frac{\partial}{\partial t}\varphi=t\log\frac{\omega_\varphi(t)^n}{\omega(t)^n}.
    \end{aligned}
\end{equation*}
we conclude that
\begin{equation}\label{obstruction flow}
    \frac{\partial}{\partial \tau}\psi(x,\tau)=\log\frac{\omega_\psi(\tau)^n}{\omega^n}(x)+\frac{1}{2}X\cdot\psi(x,\tau)-\psi(x,\tau),
\end{equation}
where $(x,\tau)\in M\times(-\infty,\log T)$, $\omega$ is the K\"ahler form of the soliton metric in Theorem \ref{main}, $\omega_\psi(\tau):=\omega+i\partial\bar{\partial}\psi(\tau)$ which satisfies $e^\tau\Phi_{e^\tau}^*\omega_\psi(\tau)=\omega_\varphi(e^\tau)$. We also denote by $g_\psi(\tau)$ the Riemannian metric induced by $\omega_\psi(\tau)$.

We call this equation \eqref{obstruction flow} the \emph{normalized K\"ahler-Ricci flow}.
\begin{lemma}[Correspondences]\label{correspondce}
    One has the following correspondences for quantities of normalized K\"ahler-Ricci flow and K\"ahler-Ricci flow: for any $(x,\tau)\in M\times(-\infty,\log T)$,
    \begin{enumerate}
        \item $e^\tau\Phi_{e^{\tau}}^*\psi(x,\tau)=\varphi(x,e^\tau)$,
        \item $e^\tau\Phi_{e^{\tau}}^*(X\cdot\psi)(x,\tau)=X\cdot\varphi(x,e^\tau)$,
        \item $e^\tau\Phi_{e^{\tau}}^*\dot{\psi}(x,\tau)=u(x,e^\tau)$, here $\dot{\psi}:=\frac{\partial}{\partial\tau}\psi$,
        \item $e^\tau\Phi_{e^{\tau}}^*(X\cdot\dot{\psi})(x,\tau)=X\cdot u(x,e^\tau)$,
        \item $e^\tau\Phi_{e^{\tau}}^*(i\partial\bar{\partial}\dot\psi)(x,\tau)=i\partial\bar{\partial} u(x,e^\tau)$,
        \item a Hamiltonian potential of $X$ with respect to $\omega_\psi(\tau)$ is given by $f_\psi(x,\tau):=f(x)+\frac{1}{2}X\cdot\psi(x,\tau)$.
    \end{enumerate}
\end{lemma}
\begin{proof}
    (i) and (ii) follow by definition. By recalling the definition of $u$ and \eqref{obstruction flow}, 
    \begin{equation*}
        \begin{split}
           & u(x,e^\tau)=e^\tau\log\frac{\omega_\varphi(e^\tau)^n}{\omega(e^\tau)^n}(x)+\frac{1}{2}X\cdot\varphi(x,e^\tau)-\varphi(x,e^\tau),\\
           &\dot{\psi}(x,\tau)=\frac{\partial}{\partial \tau}\psi(x)=\log\frac{\omega_\psi(\tau)^n}{\omega^n}(x,\tau)+\frac{1}{2}X\cdot\psi(x,\tau)-\psi(x,\tau),
        \end{split}
    \end{equation*} we conclude that $e^\tau\Phi_{e^{\tau}}^*\dot{\psi}(x,\tau)=u(x,e^\tau)$. Therefore, (iv) and (v) follow immediately.

    By the definition of $\omega_\psi(\tau)$, a Hamiltonian potential of $X$ with respect to $\omega_\varphi(e^\tau)$ is given by $e^\tau\Phi_{e^\tau}^*f+\frac{1}{2}X\cdot \varphi(e^\tau)$; hence, a Hamiltonian potential of $X$ with respect to $\omega_\psi(\tau)$ is given by $f_\psi(x,\tau):=f(x)+\frac{1}{2}X\cdot\psi(x,\tau)$.
\end{proof}
\subsection{Elliptic estimates on the sub-level set of soliton potential}

Due to Proposition \ref{resume of estimate}, we have an a priori estimate for these newly defined functions at spatial infinity.
\begin{lemma}\label{a priori estimate for psi}
    There exist constants $\lambda>0$, $C>0, A_0>1, N>0$ and $\{C_k>0\}_{k\in\N_0}$ independent of time such that for any $(x,\tau)\in\{f\ge\lambda\}\times(-\infty,\log T)$,
    \begin{enumerate}
        \item $|\dot{\psi}(x,\tau)|\le Ce^{-f(x)}f(x)^{-n-1}$,
        \item $|\psi(x,\tau)|\le Ce^{-f(x)}f(x)^{-n-1}$,
        \item $|X\cdot\psi(x,\tau)|\le C_kf(x)^{-k}$ for all $k\in\N_0$,
        \item $|X\cdot \dot\psi(x,\tau)|\le Nf(x)^{-1}$,
        \item $|i\partial\bar{\partial}\dot\psi(x,\tau)|_{\omega_\psi(\tau)}\le Nf(x)^{-1}$,
        \item $A_0^{-1}g\le g_\psi(x,\tau)\le A_0g$, and $A_0^{-1}g_0\le g_\psi(x,\tau)\le A_0g_0$.
    \end{enumerate}
    Moreover, the curvature condition in Theorem \ref{main} becomes:
    \begin{equation}\label{curvature bound}
        \sup_M|\Ric(g_\psi)(\tau)|_{g_\psi(\tau)}\le A,
    \end{equation}
    for all $\tau\in(-\infty,\log T)$.
\end{lemma}
\begin{proof}
    The result follows directly from Lemma \ref{correspondce} and Proposition \ref{resume of estimate}. For (vi), this is due to the fact that $e^{\tau}\Phi_{e^{\tau}}^*g_0=g_0$ for any $\tau\in\R$.
\end{proof}
By \cite[Proposition 2.19]{Conlon-Deruelle-Sun}, $f$ is a proper function, thus $\{f\le\lambda\}$ is a compact set, the following results establish some estimates in the compact set $\{f\le\lambda\}$.
\begin{prop}[Boundedness of $X\cdot\psi$]\label{boundedness of potential}
    There exists a constant $C_f>0$ independent of time such that $|X\cdot\psi|\le C_f$ on $\{f\le\lambda\}\times (-\infty,\log T)$.
\end{prop}
\begin{proof}
For all $\tau\in (-\infty,\log T)$,
    we observe that \begin{equation*}
        X=\nabla^{g_\psi(\tau)}\left(f+\frac{1}{2}X\cdot\psi(\tau)\right)=\nabla^gf,
    \end{equation*}
       holds on $M$. Assume that at $x_0$(resp. $y_0$) $\in\{f\le\lambda\}$, $f+\frac{1}{2}X\cdot\psi(\tau)$ achieves its maximum (resp. minimum) on $\{f\le\lambda\}$.

         If $x_0\in \{f<\lambda\}$, which means that $x_0$ lies in the interior of $\{f\le\lambda\}$, then we have $0=\nabla^{g_\psi(\tau)}(f+\frac{1}{2}X\cdot\psi)(x_0,\tau)=X(x_0)$; hence, $(f+\frac{1}{2}X\cdot\psi)(x_0,\tau)=f(x_0)\le\lambda.$ If $x_0\in\{f=\lambda\}$, then by the estimate above (Lemma \ref{a priori estimate for psi}) we have $(f+\frac{1}{2}X\cdot\psi)(x_0,\tau)\le \lambda+C_0$. Therefore, for any $x\in\{f\le\lambda\}$, we have
         \begin{equation*}
             \frac{1}{2}X\cdot\psi(x,\tau)=\left(f+\frac{1}{2}X\cdot\psi\right)(x,\tau)-f(x)\le\max_{f\le\lambda}\left(f+\frac{1}{2}X\cdot\psi(\tau)\right)-\min_{f\le\lambda}f:= C_{f,1}.
         \end{equation*}
         If $y_0\in \{f<\lambda\}$, which means $y_0$ is in the interior of $\{f\le\lambda\}$, then we have $0=\nabla^{g_\psi(\tau)}(f+\frac{1}{2}X\cdot\psi)(y_0,\tau)=X(y_0)$, thus $(f+\frac{1}{2}X\cdot\psi)(y_0,\tau)=f(y_0)\ge\min_{f\le\lambda}f.$ If $y_0\in\{f=\lambda\}$, then by the estimate above (Lemma \ref{a priori estimate for psi}) we have $(f+\frac{1}{2}X\cdot\psi)(y_0,\tau)\ge \lambda-C_0$. Therefore, for any $x\in\{f\le\lambda\}$, we have
         \begin{equation*}
             \frac{1}{2}X\cdot\psi(x,\tau)=\left(f+\frac{1}{2}X\cdot\psi\right)(x,\tau)-f(x)\ge\min_{f\le\lambda}\left(f+\frac{1}{2}X\cdot\psi(\tau)\right)-\max_{f\le\lambda}f:= -C_{f,2}.
         \end{equation*}
         Since this constant does not depend on $\tau$, on $\{f\le\lambda\}\times (-\infty,\log T)$ we have 
        \begin{equation*}
            |X\cdot\psi|\le 2\max \{C_{f,1},C_{f,2}\}:=C_f,
        \end{equation*}
        as expected.
\end{proof}
From this Proposition \ref{boundedness of potential}, the two weighted functions $e^{f_\psi}$ and $e^f$ are equivalent along the flow. The next proposition will imply that $\frac{\omega_\psi^n}{\omega^n}$ is bounded from below, it follows from the observation that $\psi+\dot\psi$ is a subsolution of the backward drift Laplacian operator $\Delta_{\omega_\psi}-\frac{X}{2}$.
\begin{prop}[Lower bound of $\psi+\dot{\psi}$]\label{boundedness of 1+2}
     There exists a constant $C>0$ independent of time such that $\dot{\psi}+\psi\ge-C$ on $\{f\le\lambda\}\times(-\infty,\log T)$.
\end{prop}
\begin{proof}
    For $\tau\in (-\infty,\log T)$, since $g$ is an expanding soliton, $\mathcal{L}_{\frac{X}{2}}g-\Ric(g)-g=0$, and by \eqref{obstruction flow},
    \begin{equation*}
        \begin{aligned}
            i\partial\bar{\partial}\dot{\psi}(\tau)&=\Ric(\omega)-\Ric(\omega_\psi(\tau))+\mathcal{L}_{\frac{X}{2}}(\omega_\psi(\tau)-\omega)+\omega-\omega_\psi(\tau)\\
            &=\mathcal{L}_{\frac{X}{2}}(\omega_\psi(\tau))-\Ric(\omega_\psi(\tau))-\omega_\psi(\tau).
        \end{aligned}
    \end{equation*}
    Thus, by taking the trace of the previous identity with respect to metric $\omega_\psi(\tau)$, we get
         \begin{equation*}
             \Delta_{\omega_\psi(\tau)}\dot{\psi}=\tr_{\omega_\psi(\tau)}(\mathcal{L}_{\frac{X}{2}}\omega_\psi(\tau))-R_{\omega_\psi(\tau)}-n.
         \end{equation*}
         Moreover, observe that
         \begin{equation*}
             \Delta_{\omega_\psi(\tau)}\psi=\tr_{\omega_\psi(\tau)}(\omega_\psi(\tau)-\omega)=n-\tr_{\omega_\psi(\tau)}\omega.
         \end{equation*}
        By \eqref{obstruction flow},
         \begin{equation*}
             \frac{X}{2}\cdot\dot{\psi}=\tr_{\omega_\psi(\tau)}(\mathcal{L}_{\frac{X}{2}}\omega_\psi(\tau))-\tr_\omega(\mathcal{L}_{\frac{X}{2}}\omega)+\frac{X}{2}\cdot\left(\frac{X}{2}\cdot\psi\right)-\frac{X}{2}\cdot\psi.
         \end{equation*}
         Since $X$ is real-holomorphic and $JX\cdot\psi=0$, $\frac{X}{2}\cdot(\frac{X}{2}\cdot\psi)=\frac{1}{2}g_\psi(\tau)(X,X)-\frac{1}{2}g(X,X)\ge -|\partial f|^2_g$. Combining this with soliton equation \eqref{soliton2} and Proposition \ref{boundedness of potential} gives:
         \begin{equation*}
             \begin{aligned}
                 \Delta_{\omega_\psi(\tau)}(\dot{\psi}+\psi)-\frac{X}{2}(\dot{\psi}+\psi)&\le |\partial f|^2_g+\tr_\omega(\mathcal{L}_{\frac{X}{2}}\omega)-R_{\omega_\psi(\tau)}-\tr_{\omega_\psi(\tau)}w\\
             &\le |\partial f|^2_g+n+R_w-R_{\omega_\psi(\tau)}-n\left(\frac{\omega^n}{\omega_\psi(\tau)^n}\right)^{\frac{1}{n}}\\
             &=|\partial f|^2_g+n+R_w-R_{\omega_\psi(\tau)}-ne^{-\frac{1}{n}(\dot{\psi}+\psi-\frac{1}{2}X\cdot\psi)}\\
             &=f+n-R_{\omega_\psi(\tau)}-ne^{-\frac{1}{n}(\dot{\psi}+\psi-\frac{1}{2}X\cdot\psi)}.\\
             \end{aligned}
         \end{equation*}
where the second line is ensured by the arithmetic-geometric mean inequality. The third line comes from \eqref{obstruction flow}, the fourth line comes from Lemma \ref{soliton equalities}. Thanks to the curvature condition (Lemma \ref{a priori estimate for psi}), on $\{f\le\lambda\}\times (-\infty,\log T)$, we have
\begin{equation*}
   \begin{split}
        \Delta_{\omega_\psi(\tau)}(\dot{\psi}+\psi)-\frac{X}{2}(\dot{\psi}+\psi)&\le f+n-R_{\omega_\psi(\tau)}-ne^{-\frac{1}{n}(\dot{\psi}+\psi-\frac{1}{2}X\cdot\psi)}\\&\le \lambda+n+A-ne^{-\frac{1}{n}(\dot{\psi}+\psi-\frac{1}{2}X\cdot\psi)}.
   \end{split}
\end{equation*}
         Assume that at $x_0$, the function $\dot{\psi}+\psi$ achieves its minimum. If $x_0\in\{f<\lambda\}$, the minimum principle implies that 
         \begin{equation*}
             0\le \left(\Delta_{\omega_\psi}(\dot{\psi}+\psi)-\frac{X}{2}(\dot{\psi}+\psi)\right)(x_0)\le \lambda+n+A-ne^{-\frac{1}{n}(\dot{\psi}+\psi-\frac{1}{2}X\cdot\psi)}(x_0).
         \end{equation*}
         This gives $(\dot{\psi}+\psi-\frac{1}{2}X\cdot\psi)(x_0)\ge-C$ for some constant independent of $\tau$ or $x_0$. Moreover, thanks to Proposition \ref{boundedness of potential}, we conclude that $(\dot{\psi}+\psi)(x_0)\ge-C$ for some constant $C>0$ independent of time and $x_0$.

         If $x_0\in\{f=\lambda\}$, we use Lemma \ref{a priori estimate for psi} to get such time independent constant $C$.
\end{proof}
Thanks to the lower bound of $\psi+\dot\psi$, and the algebraic structure of $\psi+\dot\psi$, we get a lower bound of $\psi$.
\begin{prop}[Lower bound of $\psi$]\label{lower bound of psi}
         There exists a constant $C>0$ independent of time such that $\psi>-C$ on $\{f\le\lambda\}\times(-\infty,\log T)$.
     \end{prop}
     \begin{proof}
         Pick $x\in\{f\le\lambda\}$ and $\tau\in(-\infty,\log T)$ outside the exceptional set $E$. Set $y=\Phi_{\frac{1}{t}}(x)$, here $t=e^\tau>0$. By \cite[Lemma 2.6]{Conlon-Deruelle-Sun} $y\notin E$ if $x\notin E$. By the initial normalization of the K\"ahler potential as in Proposition \ref{Proposition monge ampere} we have
         \begin{equation*}
             \lim_{s\to 0^+}\varphi(y,s)=0,
         \end{equation*}
         Recall that by the definition of the K\"ahler potential $\psi$,
         \begin{equation*}
             e^\rho\psi(\Phi_s(y),\rho)=\varphi(y,s),\quad s=e^\rho.
         \end{equation*}
         Hence we have
         \begin{equation*}
             \lim_{\rho\to-\infty}e^\rho\psi(\Phi_s(y),\rho)=0.
         \end{equation*}
        It follows from Lemma~\ref{a priori estimate for psi}, Proposition~\ref{boundedness of potential}, and Proposition~\ref{boundedness of 1+2} that the function
\[
\psi+\dot{\psi}-\frac{1}{2}X\cdot\psi
\]
is uniformly bounded below on $M\times(-\infty,\log T)$.
         Therefore there exists a time-independent constant $C>0$ such that
        \begin{equation*}
            \begin{aligned}
                \partial_\rho(e^\rho\psi(\Phi_s(y),\rho))&=e^\rho(\psi+\dot{\psi})(\Phi_s(y),\rho)-e^\rho\frac{1}{2e^\rho}X\cdot\psi(\Phi_s(y),\rho)\frac{\partial s}{\partial\rho}\\
             &=e^\rho(\psi+\dot{\psi}-\frac{1}{2}X\cdot\psi)(\Phi_s(y),\rho)\\
             &\ge -Ce^\rho.
            \end{aligned}
        \end{equation*}
     Since $\lim_{\rho\to-\infty}e^\rho\psi(\Phi_s(y),\rho)=0,$ integration gives
        \begin{equation*}
            e^\tau\psi(x,\tau)=e^\tau\psi(\Phi_t(y),\tau)\ge-Ce^\tau,
        \end{equation*}
         hence $\psi(x,\tau)\ge-C$ in $\{f\le\lambda\}\times(-\infty,\log T)$. Since the complement of $E$ is dense, by continuity of $\psi$, we have $\psi\ge-C$ in $\{f\le\lambda\}\times(-\infty,\log T)$.
     \end{proof}
      \begin{prop}\label{total bound of psi}
      There exists a time independent constant $C>0$ such that
         $|\psi|+|\dot{\psi}|\le C$ on $\{f\le\lambda\}\times(-\infty,\log T)$.
     \end{prop}
     \begin{proof}
         Thanks to Corollary \ref{subharmonic function}, there is an $\varepsilon>0$ such that \begin{equation} \label{subharmonic}
             \Delta_\omega f\ge \varepsilon>0,\quad \textrm{on $M$}.
         \end{equation}
Notice that for $\tau\in(-\infty,\log T)$ by \eqref{obstruction flow}
\begin{equation*}
    i\partial\bar{\partial}\dot{\psi}=i\partial\bar{\partial}f_\psi-\Ric({\omega_\psi(\tau)})-\omega-i\partial\bar{\partial}\psi,
\end{equation*}
thus by taking the trace with respect to the soliton metric $w$, we have
\begin{equation*}
    \Delta_\omega(\dot{\psi}+\psi-f_\psi)=-\tr_\omega\Ric({\omega_\psi(\tau)})-n.
\end{equation*}
Because we have assumed that $\Ric({\omega_\psi(\tau)})\le A\omega_\psi(\tau)$, we deduce that
\begin{equation*}
    \Delta_\omega(\dot{\psi}+\psi-f_\psi)\ge-A\tr_w\omega_\psi(\tau)-n=-(A+1)n-A\Delta_\omega\psi,
\end{equation*}
and,
\begin{equation*}
    \Delta_{\omega}(\dot{\psi}+(A+1)\psi-f_\psi)\ge-(A+1)n.
\end{equation*}
By \ref{subharmonic},
\begin{equation*}
    \Delta_\omega\left(\dot{\psi}+(A+1)\psi-f_\psi+\frac{(A+2)n}{\varepsilon}f\right)\ge -(A+1)n+\varepsilon\frac{(A+2)n}{\varepsilon}=n>0.
\end{equation*}
Thus by the weak maximum principle, we have that 
\begin{equation*}
    \max_{\{f\le\lambda\}}\left(\dot{\psi}+(A+1)\psi-f_\psi+\frac{(A+2)n}{\varepsilon}f\right)=\max_{\{f=\lambda\}}\left(\dot{\psi}+(A+1)\psi-f_\psi+\frac{(A+2)n}{\varepsilon}f\right),
\end{equation*}
hence we have that  $\dot{\psi}+(A+1)\psi$ is uniformly bounded from above on $\{f\le\lambda\}$ by a constant independent of time thanks to Lemma \ref{a priori estimate for psi}.

Because $\psi$ and $\dot{\psi}+\psi$ are bounded from below (Proposition \ref{boundedness of 1+2} and Proposition \ref{lower bound of psi}), we have that $|\psi|+|\dot{\psi}|$ is uniformly bounded by a constant $C>0$ independent of time on $\{f\le\lambda\}$.
     \end{proof}
     From Proposition \ref{boundedness of potential} and Proposition \ref{total bound of psi}, we know that the two volume forms $\omega^n$ and $\omega_\psi(\tau)^n$ are equivalent along the normalized K\"ahler-Ricci flow. This allows us to construct an energy functional to prove Theorem \ref{main}.
\subsection{Proof of the Uniqueness theorem}
In this section, we provide a detailed proof of Theorem \ref{main} using an energy method.
\begin{proof}[Proof of Theorem \ref{main}]
     Choose $k\in\mathbb{N^*}$ such that $k\ge2$ and $2k-1\ge|\dot{\psi}|$ on $M\times (-\infty,\log T)$. The existence of $k$ is ensured by Lemma \ref{a priori estimate for psi} and Proposition \ref{total bound of psi}. We define an energy $A_k(\tau)$ for $\tau\in (-\infty,\log T)$ as:
     \begin{equation*}
             A_k(\tau):=\int_M\dot{\psi}(\tau)^{2k}e^{f_\psi(\tau)}\omega_\psi(\tau)^n.
         \end{equation*}
         For convenience, we omit the time dependence $\tau$, and we define $\Delta_{\omega_\psi,X}:=\Delta_{\omega_\psi}+\frac{1}{2}X$.
         \begin{claim}\label{1}
           The functional $A_k(\tau)$ is well defined for all $\tau\in (-\infty,\log T)$, and there exists a constant $B>0$ independent of time such that $A_k(\tau)\le B$ for all $\tau\in (-\infty,\log T)$. 
         \end{claim}
         \begin{proof}[Proof of Claim \ref{1}]
                By Lemma \ref{a priori estimate for psi}, there exists a constant $C>0$ such that for all $\tau\in (-\infty,\log T)$, $|\dot{\psi}(\tau)|\le Ce^{-f}f^{-n-1}$ on $\{f\ge\lambda \}$, thus $\dot\psi(\tau)^{2k}\le C^{2k}(e^{-f}f^{-n-1})^{2k}$ on $\{f\ge\lambda\}$. By Lemma \ref{a priori estimate for psi}, $\omega_\psi^n\le A_0 \omega_0^n$, $\omega_0$ being the K\"ahler form of $g_0$ on $\{f\ge\lambda\}$. Hence on $\{f\ge\lambda\}$, by Lemma \ref{a priori estimate for psi}, there exists a constant $C>0$ such that
    \begin{equation*}
        \dot\psi^{2k}e^{f_\psi}\omega_\psi^n\le C(e^{-f}f^{-n-1})^{2k}e^{f}\omega_0^n.
    \end{equation*}
    Since at infinity $f$ and $f_\psi$ are equivalent to $\frac{r^2}{2}$, $r$ being the radial function, we deduce that if $k\ge 2$, $A_k(\tau)$ is well defined, and there exists a constant $B_1>0$ independent of time such that 
    \begin{equation*}
        \int_{\{f>\lambda\}}\dot{\psi}^{2k}e^{f_\psi}\omega_\psi^n\le B_1.
    \end{equation*}
    Moreover, on $\{f\le \lambda\}$, by Proposition \ref{boundedness of potential} and Proposition \ref{total bound of psi}, $|\dot\psi|+e^{f_\psi}\le C$ holds uniformly in time. Since $\log\frac{\omega_\psi^n}{\omega^n}=\dot\psi+\psi-\frac{1}{2}X\cdot\psi$, as a consequence of Proposition \ref{boundedness of potential} and \ref{total bound of psi}, $\omega_\psi^n\le C\omega^n$ holds uniformly in time on $\{f\le\lambda\}$ . Therefore, there exist constants $C>0$ and $B_2>0$ such that
    \begin{equation*}
        \int_{\{f\le\lambda\}}\dot{\psi}^{2k}e^{f_\psi}\omega_\psi^n\le C\int_{\{f\le\lambda\}}\omega^n\le B_2.
    \end{equation*}
    One can take $B=B_1+B_2$ as the time independent constant as required. 
         \end{proof}
         \begin{claim}\label{2}
     The functional $A_k(\tau)$ is continuously differentiable, and 
           \begin{equation}\label{derivative}
         \partial_\tau A_k(\tau)=2k\int_M(\Delta_{\omega_\psi,X}\dot{\psi}-\dot{\psi})\dot{\psi}^{2k-1}e^{f_\psi}\omega_\psi^n+\int_M\dot{\psi}^{2k}\Delta_{\omega_\psi,X}\dot{\psi}e^{f_\psi}\omega_\psi^n.
    \end{equation}
         \end{claim}
          \begin{proof}[Proof of Claim \ref{2}]
              Thanks to \eqref{obstruction flow}, one has the following pointwise computation:
              \begin{equation}\label{evolution equation of dotpsi}
                  \frac{\partial}{\partial\tau}\dot\psi=\Delta_{\omega_\psi}\dot\psi+\frac{X}{2}\cdot\dot\psi-\dot\psi=\Delta_{\omega_\psi,X}\dot\psi-\dot\psi.
              \end{equation}
              and
             \begin{equation}\label{evolution equation of vulome ration}
                 \frac{\partial}{\partial\tau}(e^{f_\psi}\omega_\psi^n)=\left(\Delta_{\omega_\psi}\dot\psi+\frac{1}{2}X\cdot\dot\psi\right)e^{f_\psi}\omega_\psi^n=e^{f_\psi}\omega_\psi^n\Delta_{\omega_\psi,X}\dot\psi .
             \end{equation}
             Combining \eqref{evolution equation of dotpsi} and \eqref{evolution equation of vulome ration}, we have,
               \begin{equation*}
                  \frac{\partial}{\partial\tau}( \dot{\psi}^{2k}e^{f_\psi}\omega_\psi^n)=\left(2k\dot\psi^{2k-1}(\Delta_{\omega_\psi,X}\dot\psi-\dot\psi)+\dot\psi^{2k}\Delta_{\omega_\psi,X}\dot\psi\right)e^{f_\psi}\omega_\psi^n.
               \end{equation*}
               At spatial infinity, $\Delta_{\omega_\psi,X}\dot\psi$ is uniformly bounded by a constant independent of time due to Lemma \ref{a priori estimate for psi}. Thus for any $\tau\in (-\infty,\log T)$ there is a time independent constant $C'$ such that on $\{f>\lambda\}$:
      \begin{equation*}
        |2k\dot\psi^{2k-1}(\Delta_{\omega_\psi,X}\dot\psi-\dot\psi)+\dot\psi^{2k}\Delta_{\omega_\psi,X}\dot\psi|\le C'C^{2k}(e^{-f}f^{-n-1})^{2k}+C'C^{2k-1}(e^{-f}f^{-n-1})^{2k-1}.
    \end{equation*}
    RHS is integrable at infinity, and by Lemma \ref{a priori estimate for psi}, Proposition \ref{boundedness of potential} and Proposition \ref{total bound of psi}, the measure $e^{f_\psi}\omega_\psi^n$ is equivalent to $e^f\omega^n$. Hence, by the dominated convergence theorem, we have that \eqref{derivative} holds as expected.
           \end{proof}  
           \begin{claim}[Divergence theorem]\label{Divergence theorem}
    Let $m\ge 3$ be an odd integer, then \begin{equation}
        \int_M \dot\psi^{m}\Delta_{\omega_\psi,X}\dot\psi e^{f_\psi}\omega_\psi^n=-m\int_{M}|\partial\dot\psi|_{\omega_\psi}^2\dot\psi^{m-1}e^{f_\psi}\omega_\psi^n,
    \end{equation}
    holds for all $\tau\in (-\infty,\log T)$
\end{claim}
   \begin{proof}[Proof of Claim \ref{Divergence theorem}]
      
For $R>0$ large enough, the level set $\{f_\psi=R\}$ is a smooth submanifold of $M$, then by divergence theorem:
\begin{equation*}
    \int_{\{f_\psi\le R\}}\dot\psi^{m}\Delta_{\omega_\psi,X}\dot\psi e^{f_\psi}\omega_\psi^n+m\int_{\{f_\psi\le R\}}|\partial\dot\psi|_{\omega_\psi}^2\dot\psi^{m-1}e^{f_\psi}\omega_\psi^n=\frac{n!}{2}\int_{\{f_\psi=R\}}\dot\psi^{m}\frac{\partial\dot\psi}{\partial\nu}e^{f_\psi}d\sigma.
\end{equation*}
Here $\nu$ is the outward-pointing unit normal vector to $\{f_\psi=R\}$ with respect to metric $\omega_\psi$, and $d\sigma$ is the induced measure on $\{f_\psi=R\}$. Since $f_\psi$ is a Hamiltonian potential of $X$ with respect to the metric $\omega_\psi$, we have $\nu=\frac{X}{|X|_{\omega_\psi}}$ and $|\frac{\partial\dot\psi}{\partial\nu}|=\frac{|X\cdot\dot\psi|}{|X|_{\omega_\psi}}$. As a consequence of estimates in Lemma \ref{a priori estimate for psi}, in particular,  we have that $|\frac{\partial\dot\psi}{\partial\nu}|$ is bounded on $M$. If $m\ge 3$, by the same argument as in Claim \ref{2}, we have
\begin{equation*}
    \lim_{R\to+\infty}\int_{\{f_\psi=R\}}\dot\psi^{m}\frac{\partial\dot\psi}{\partial\nu}e^{f_\psi}d\sigma=0.
\end{equation*}
Moreover, since $m$ is odd, $|\partial\dot\psi|_{\omega_\psi}^2\dot\psi^{m-1}e^{f_\psi}$ is nonnegative. Letting $R$ tend to $+\infty$ we have that $|\partial\dot\psi|_{\omega_\psi}^2\dot\psi^{m-1}e^{f_\psi}$ is Lebesgue integrable with respect to $\omega_\psi^n$ and
\begin{equation*}
    \int_M\dot\psi^{m}\Delta_{\omega_\psi,X}\dot\psi e^{f_\psi}\omega_\psi^n+m\int_M|\partial\dot\psi|_{\omega_\psi}^2\dot\psi^{m-1}e^{f_\psi}\omega_\psi^n=0,
\end{equation*}
holds as expected for all $\tau\in (-\infty,\log T)$.
   \end{proof}
    \begin{claim}\label{3}
       The function $\dot{\psi}$ is identically 0.
     \end{claim}
     \begin{proof}[Proof of Claim \ref{3}]
         By Claim \ref{1}, $A_k(\tau)$ is a well-defined nonnegative energy and there exists a uniform constant $B>0$ such that $A_k(\tau)\le B,\forall\tau\in(-\infty,\log T)$. By Claim \ref{2} and Claim \ref{Divergence theorem}
         \begin{equation*}
             \begin{aligned}
                 \partial_\tau A_k(\tau)&=2k\int_M(\Delta_{\omega_\psi,X}\dot{\psi}-\dot{\psi})\dot{\psi}^{2k-1}e^{f_\psi}\omega_\psi^n+\int_M\dot{\psi}^{2k}\Delta_{\omega_\psi,X}\dot{\psi}e^{f_\psi}\omega_\psi^n\\
             &=-2kA_k(\tau)-2k(2k-1)\int_M|\partial\dot{\psi}|_{\omega_\psi}^2\dot{\psi}^{2k-2}e^{f_\psi}\omega_\psi^n+\int_M\dot{\psi}^{2k}\Delta_{\omega_\psi,X}\dot{\psi}e^{f_\psi}\omega_\psi^n\\
             \end{aligned}
         \end{equation*}
         For $R>0$ large enough, as in the proof of Claim \ref{Divergence theorem}, by the divergence theorem, we have,
         \begin{equation*}
              \int_{\{f_\psi\le R\}}\dot\psi^{2k}\Delta_{\omega_\psi,X}\dot\psi e^{f_\psi}\omega_\psi^n+2k\int_{\{f_\psi\le R\}}|\partial\dot\psi|_{\omega_\psi}^2\dot\psi^{2k-1}e^{f_\psi}\omega_\psi^n=\frac{n!}{2}\int_{\{f_\psi=R\}}\dot\psi^{2k}\frac{\partial\dot\psi}{\partial\nu}e^{f_\psi}d\sigma,
         \end{equation*}
         and
         \begin{equation*}
              \lim_{R\to+\infty}\int_{\{f_\psi=R\}}\dot\psi^{2k}\frac{\partial\dot\psi}{\partial\nu}e^{f_\psi}d\sigma=0.
         \end{equation*}
         Since $2k-1\ge |\dot\psi|$ holds uniformly in time, we get
            \begin{equation*}
              \int_{\{f_\psi\le R\}}\dot\psi^{2k}\Delta_{\omega_\psi,X}\dot\psi e^{f_\psi}\omega_\psi^n\le \int_{\{f_\psi=R\}}\dot\psi^{2k}\frac{\partial\dot\psi}{\partial\nu}e^{f_\psi}d\sigma+2k(2k-1)\int_{\{f_\psi\le R\}}|\partial\dot\psi|_{\omega_\psi}^2\dot\psi^{2k-2}e^{f_\psi}\omega_\psi^n,
         \end{equation*}
         Due to Claim \ref{Divergence theorem}, $|\partial\dot\psi|_{\omega_\psi}^2\dot\psi^{2k-2}e^{f_\psi}$ is Lebesgue integrable with respect to $\omega_\psi^n$ for $k\ge 2$. By letting $R$ tend to $+\infty$, we conclude that
         \begin{equation*}
             \int_M\dot\psi^{2k}\Delta_{\omega_\psi,X}\dot\psi e^{f_\psi}\omega_\psi^n\le 2k(2k-1)\int_M|\partial\dot\psi|_{\omega_\psi}^2\dot\psi^{2k-2}e^{f_\psi}\omega_\psi^n.
         \end{equation*}
         This implies that
         \begin{equation*}
             \partial_\tau A_k(\tau)\le -2kA_k(\tau).
         \end{equation*}
        and that $e^{2k\tau}A_k(\tau)$ is non increasing. For all $\rho\le \tau$, by integration and Claim \ref{1}, we have
         \begin{equation*}
             e^{2k\tau}A_k(\tau)\le e^{2k\rho}A_k(\rho)\le Be^{2k\rho}.
         \end{equation*}
       By letting $\rho$ tend to $-\infty$, we get $A_k(\tau)=0$ for any $\tau$, thus $\dot{\psi}$ is identically 0. 
     \end{proof}
     Once we have $\dot\psi=0$, we deduce that $\psi=0$.
     \begin{claim}\label{4}
       The function $\psi$ is identically 0. Hence $g_\varphi(t)=g(t)$ for all $t\in (0,T)$.
     \end{claim}
     \begin{proof}[Proof of Claim \ref{4}]
          Since $\dot{\psi}$ is identically 0, equation \eqref{obstruction flow} reduces to
 \begin{equation*}
     \log\frac{\omega_\psi^n}{\omega^n}+\frac{X}{2}\cdot\psi-\psi=0.
 \end{equation*}
 Fix $\tau\in (-\infty,\log T)$. By Lemma \ref{a priori estimate for psi}, $|\psi(\tau)|$ tends to 0 at spatial infinity. If $\sup_M\psi(\tau)>0$, there exists $x_0\in M$ such that $\psi(x_0,\tau)=\sup_M\psi(\tau)>0$. Due to the maximum principle, at the point $x_0$:
 \begin{equation*}
     0\ge\log\frac{\omega_\psi(\tau)^n}{\omega^n}(x_0)+\frac{X}{2}\cdot\psi(x_0)=\psi(x_0).
 \end{equation*}
 This gives a contradiction. Hence, $\sup_M\psi(\tau)\le0$. The analogous minimum argument
gives $\inf_M\psi(\tau)\ge 0$; we have $\psi(\tau)=0$ as desired.
     \end{proof}
\end{proof}
\section{Open questions}
     This section lists several open questions for further study.
     \begin{ques}
         Does the Ricci flow evolve uniquely after a singularity? If not, is
there a selection principle for the best flow?
     \end{ques}
     This question was already posted in \cite{Feldman-Ilmanen-Knopf}. This paper did not provide a full answer to this question. For noncompact manifolds, the answer is negative, as shown by the results of Angenent and Knopf \cite{Angenent-Knopf}. For compact manifolds with isolated conical singularities, we refer to \cite{Gianniotis-Schulze} and \cite{CHL}.
\begin{ques}
     Let $g_1(t)_{t\in (0,T)}$ be a smooth complete solution to K\"ahler-Ricci flow, such that $\pi_*g_1(t)$ converges to $g_0$ locally smoothly when $t$ tends to 0, is it true that $JX$ is a Killing vector field for $g_1(t)$ for any $t\in (0,T)$?
\end{ques}
In Theorem \ref{main}, we assume that $ JX $ is a Killing vector field with respect to $ g_\varphi(t_0) $ for some $ t_0 \in (0, T) $. Given that $ J_0X $ is Killing with respect to $ g_0 $, we aim to determine whether $ JX $ remains Killing for $ g_1(t) $ for all $ t \in (0, T) $.
\begin{ques}
     Let $g_1(t)_{t\in (0,T)}$ be a smooth complete solution to K\"ahler-Ricci flow, such that $\pi_*g_1(t)$ converges to $g_0$ locally smoothly when $t$ tends to 0. Is there any algebraic obstruction of the exceptional set $E$ to ensure that $\omega_1(t)-\omega(t)$ is globally exact on $M$?
\end{ques}
In Remark \ref{rmk}, we have noticed that $\pi_*(\omega_1(t)-\omega(t))$ is exact on $C\setminus\{o\}$. But to obtain that $\omega_1(t)-\omega(t)$ is globally exact on $M$, we will need more information about this exceptional set $E$. We want to know if there is any algebraic obstruction of $E$ to ensure that $\omega_1(t)-\omega(t)$ is exact on $M$.
\begin{ques}
    Is the curvature condition sharp in Theorem \ref{main}?
\end{ques}
This question is about the non-uniqueness problem of K\"ahler-Ricci flow. Here one wants to know whether there exists an asymptotically conical solution to K\"ahler-Ricci flow different from $g(t)$ with unbounded Ricci curvature.

We also consider these questions in a more general setting, namely, for solutions to the Ricci flow on Riemannian manifolds that are AC. A trivial example of an AC expanding Ricci soliton is the Gaussian soliton $(\R^n,g_{\operatorname{eucl}},\frac{1}{2}r\partial_r)$, the self-similar solution to the Ricci flow is $g(t)=g_{\operatorname{eucl}}$.

Another example is the Bryant soliton (see \cite[Section 5 of Chapter 1]{Chow} and \cite{Bryant}), a one-parameter family of expanding gradient Ricci solitons $(\mathbb{R}^n, g_c, \nabla^{g_c} f_c)_{c \in (0,1)}$, which are $SO(n)$-invariant and satisfy $\operatorname{Rm}(g_c) > 0$. These self-similar solutions are asymptotic to the cone $(C(\mathbb{S}^{n-1}),\, dr^2 + (cr)^2 g_{\mathbb{S}^{n-1}},\, \tfrac{1}{2}r\partial_r)$.

In \cite{Deruelle2}, it is shown that the limit $\lim_{f\to +\infty}(f+\mu(g))^{\frac{n}{2}-1}e^{f+\mu(g)}\Ric(g)$ exists for a normalized asymptotically Ricci-flat expanding gradient soliton $(M^n,g,\nabla^gf)$. Here $\mu(g)$ denotes the entropy. In \cite{Siepmann}, it is proved that $\Ric(g(t))=O(t^{-1}e^{-\frac{r^2}{4t}}(\frac{r^2}{4t})^{\kappa+1})$ with some $\kappa>0$ for a solution to Ricci flow which is asymptotic to Ricci flat cone $(C,g_C)$ with radial function $r$.
\begin{ques}
    What is the optimal rate of convergence of $\Ric(g(t))$ for a solution $g(t)_{t\in (0,T)}$ to Ricci flow which is asymptotic to a Ricci flat cone? i.e., what is the optimal $\tau\in\R$ such that on some $\Omega_\lambda$ we have
    \begin{equation*}
        \Ric(g(t))=O\left(t^{-1}e^{-\frac{r^2}{4t}}\left(\frac{r^2}{4t}\right)^{\tau}\right),\quad \forall t\in (0,T).
    \end{equation*}
\end{ques}
We conjecture that the optimal rate should be $-\frac{n}{2}+1$. If we assume that there exists a complete expanding gradient Ricci soliton that desingularize this Ricci flat cone, then with the maximum principle in Section \ref{A priori estimate at spatial infinity}, one can prove that $\Ric(g(t))=O\left(t^{-1}e^{-\frac{r^2}{4t}}\left(\frac{r^2}{4t}\right)^{\tau}\right)$ for any Ricci flow $(g(t))_{t\in (0,T)}$ that smooths out this cone.

In \cite{Feldman-Ilmanen-Knopf}, besides the expanding K\"ahler-Ricci soliton, they constructed shrinking and steady K\"ahler-Ricci soliton whose pointed Gromov-Hausdorff limit is the K\"ahler cone $\C^n/\Z^k$ with $k\le n$.
\begin{ques}\label{Ilmanen conjecture}
    Is there a solution to Ricci flow $g(t)_{t\in (0,T)}$ on a smooth manifold $M$ such that the pointed Gromov-Hausdorff limit of $(M,g(t))$, when $t$ tends to 0, is $(\C^n/\Z^k,\partial\bar\partial\left(\frac{|\cdot|^{2p}}{p}\right))$ with $k\le n$ and some $p>0$?
\end{ques}
It seems that the Euclidean case $(p=1)$ has been (at least partially) solved by M.Siepmann in his earlier work. In \cite{Ozuch-Lopez}, Ozuch and Lopez proved that there cannot be any ALE expanding soliton on the minimal resolution of $\C^2/\Gamma$ for $\Gamma\subset SU(2)$. However, this result cannot provide a negative answer to Question \ref{Ilmanen conjecture} in complex dimension 2, we notice that the action of $\Z_k$ on $\C^2$ is in $U(2)\setminus SU(2)$ for $k\ge 3$.

\newcommand{\etalchar}[1]{$^{#1}$}

\end{document}